\documentclass[a4paper,10pt]{article}
\usepackage{a4, epsf, amscd, %dsfont,
            amsfonts, amsmath, amssymb, amstext, 
            amsthm,latexsym, subfigure, epsfig, color}

\theoremstyle{plain}
\newtheorem{theorem}{Theorem}[section]
\newtheorem{lem}[theorem]{Lemma}

\newtheorem{prop}[theorem]{Proposition}

\newtheorem{cor}[theorem]{Corollary}
\newtheorem{defi}[theorem]{Definition}
\theoremstyle{definition}

\theoremstyle{remark}
\newtheorem*{rem} {Remark}

\newcommand{\ktquot}[2]{
{#1}_{{\mbox{\hspace{-3pt}\it
      /\hspace{-3pt}/}}_{\mbox{\hspace{-3pt}$#2$}}}
}

\newcommand{\modulo}[2]{
{#1}_{{\mbox{\hspace{-3pt}\it
      /\hspace{-3pt} }}_{\mbox{\hspace{-3pt}$#2$}}}
}

\newcommand{\momentMH}{\mathcal{M}(M_H)}
\newcommand{\momentMKH}{\mathcal{M}(M_{(H)})}
\newcommand{\momentMcH}{\mathcal{M}(M^c_{H})}
\newcommand{\momentMcKH}{\mathcal{M}(M^c_{(H)})}

\newcommand{\orbit}[2]{
{#1}{{\mbox{\hspace{-3pt} $\cdot$ \hspace{-3pt}}}{#2}} }

\newcommand{\longhookuparrow}{^{\cup}\hskip-5.15pt{ }^{{}^{\textstyle\Big\uparrow}}}

\newcommand{\longhookrightarrow}{\ensuremath{\lhook\joinrel\relbar\joinrel\rightarrow}}

\title{Equivariant K\"ahlerian extensions of contact manifolds}

\author{Ay\c{s}e Kurtdere}
\date{}

\begin{document}
\parindent0cm

\maketitle

\begin{abstract}
\noindent
For contact manifolds $(M, \eta)$ a complexification $M^c$ is constructed
to which the contact form $\eta$ extends such that the exterior derivative
of the extended form is K\"ahlerian. In the case of a proper action of an 
extendable Lie group this construction is realized in an equivariant way.
In a simultaneous stratification of $M$ and $M^c$ according to the 
istropy type, it is shown that the K\"ahlerian reduction of the complexification 
can be seen as the complexification of the contact reduction.
\end{abstract}

%\thispagestyle{empty}
%{\phantom{Dieser Text erscheint nicht.}}
%\vspace{15cm}

%\newpage 
%\thispagestyle{empty}
%\tableofcontents

%\newpage

%\thispagestyle{empty}
%{\phantom{Dieser Text erscheint nicht.}}
%\vspace{15cm}
%\newpage

%\setcounter{page}{0}
%\pagenumbering{Roman}

%\addcontentsline{toc}{section}{\numberline{}Einleitung}

\setcounter{page}{1}

%-----------------------------------------------------------

\bigskip

\section{Introduction}

\bigskip

Manifolds with additional structure can sometimes be understood better 
if the structure extends to a complexification of the manifold.
By a result of Whitney (\cite{Whitney-analytic-extensions}) 
and Shutrick (\cite{Shutrick})
a differentiable manifold $M$ can be embedded as 
a closed, real analytic and totally real submanifold of a complex manifold $M^c$ 
with the dimension $\dim_{\mathbb C} M^c = \dim_{\mathbb R} M$.
Using his solution of the Levi problem 
Grauert (\cite{Grauert}) proved that 
the complexification $M^c$ can be realized as a Stein manifold,
in particular,
it can be holomorphically and properly embedded in some $\mathbb C^N$.
During the last two decades, 
complexifications of real manifolds 
with additional structure
achieved some attention.
%A real analytic manifold with a proper and real analytic $G$-action 
%can be complexified according to Heinzner 
%(\cite{Heinzner-equivariant-extensions}).
%
%
An equivariant version for proper actions has been shown in 
\cite{Heinzner-equivariant-extensions}.
Stratmann (\cite{Stratmann}) considers proper actions of Lie groups $G$ 
on symplectic manifolds $(M, \omega)$
and shows that there is a Stein complexification $M^c$ of $M$ 
with a $G$-invariant 
K\"ahler form $\tau$ such that 
$\omega = \iota_M^{\ast} (\tau)$
where $\iota_M : M \hookrightarrow M^c$
is the embedding of $M$ in $M^c$.

%\bigskip

In this paper a similar extension result is shown for $1$-forms.
%
%
%The motivation of this project 
%was to extend contact structures on a manifold $M$ 
%to a suitable complexification $M^c$ under the presence of symmetries.
%More precisely, 
%the action of a Lie group $G$ is assumed to be proper 
%and to leave an arbitrary smooth $1$-form $\eta$ invariant.
%A particular case is a contact manifold $(M, \eta)$
%on which a Lie group $G$ acts properly by contact transformations.
%
Contact manifolds $(M, \eta)$
on which a Lie group $G$ acts properly by contact transformations 
are of particular interest.
As a general assumption in this work
the Lie group $G$ has finitely many connected components, 
is extendable, i.e., 
the natural homomorphism $G \rightarrow G^{\mathbb C}$ is injective, 
and is acting properly on $M$ as a group of diffeomorphisms.
%
%
%\bigskip
%
%
In Section \ref{Extension of forms} 
%and 
%\ref{Equivariant extensions in the case of proper group actions},
an equivariant complexification $\iota_M : M \hookrightarrow M^c$
such that $G$ acts on $M^c$ properly by holomorphic transformations
and a strictly plurisubharmonic, $G$-invariant function 
$\varrho : M^c \rightarrow \mathbb R$ 
are constructed with the property that
%\begin{equation}\label{iota-dc-varrho} 
%  \eta = \iota_M^{\ast} (d^c \varrho).
%\end{equation}
$  \eta = \iota_M^{\ast} (d^c \varrho).$
For this, a slice construction for $M = G \times^K S$,
$K$ maximal compact in $G$, is used to construct a complexification $M^c$
of $M$ by a complexification $G^{\mathbb C} \times S^{\mathbb C}$ of 
$G \times S$:
\[ \begin{array}{rcl} 
     G \times S & \longhookrightarrow & 
     \Omega \subset G^{\mathbb C} \times S^{\mathbb C} \\
     & & \\
     \Big \downarrow \mbox{ }^{ \pi_K } & & 
      \Big \downarrow \mbox{ }^{\pi_{K^{\mathbb C}} } \\
     & & \\
     M = G \times^K S & \longhookrightarrow & 
     M^{\mathbb C} = G^{\mathbb C} \times^{K^{\mathbb C}} S^{\mathbb C}.
   \end{array} \]
The interplay of complexifications and contact reductions 
%(\cite{Loose-contact-reduction}, \cite{Willett})
and K\"ahlerian reductions 
%(\cite{Heinzner-geometric-invariant-stein}, \cite{Sjamaar})
is discussed in the remaining section.
Roughly speaking, the complexifications of contact reductions 
of a $G$-contact manifold $(M, \eta)$ 
can be seen as the K\"ahlerian reduction of the complexification.
This is shown along a simultaneous stratification of both 
the contact manifold and its complexification.

%\bigskip
%{\phantom{Dieser Text erscheint nicht.}}
%\vspace{1cm}

\bigskip

\section{Extension of forms}\label{Extension of forms}

\bigskip

%Let $K$ be a compact Lie group 
%and let $M$ be a $K$-manifold with an invariant $1$-form $\eta$.
%The aim of the first section is 
%to construct a smooth, $K$-equivariant Stein complexification $M^c$ of $M$ 
%to which $\eta$ is extended such that
%$ \iota_M^{\ast}(d^c \varrho) = \eta,$
%where $\varrho : M^c \rightarrow \mathbb R$
%is a smooth, $K$-invariant and strictly plurisubharmonic function on $M$
%and $\iota_M : M \hookrightarrow M^c$.
%
%
%This is motivated by the question as to whether 
%the contact
%%\footnote{A contact form $\eta$ on a manifold $M$ is a $1$-form such that $ \eta \wedge (d \eta)^n $ is a (nowhere vanishing) volume form. Alternative definitions are discussed, for instance, by Blair (\cite{Blair}, Chapter 1).} 
%form $\eta$ of a contact manifold $(M, \eta)$
%can be extended to a complexification $M^c$ of $M$.
%A contact form $\eta$ on a manifold $M$ is a 
%$1$-form such that $ \eta \wedge (d \eta)^n $
%is a (nowhere vanishing) volume form.
%In the following, arbitrary $1$-forms $\eta$
%are considered;
%the contact condition is not needed.

Let $G$ be an extendable Lie group.
Let $M$ be a real analytic $G$-manifold with a $1$-form $\eta$.
%There is a $G$-equivariant complexification $M^c$ of $M$
%with real analytic $G$-equivariant map
%$\iota_M : M \rightarrow M^c$
%(\cite{Heinzner-equivariant-extensions}, \cite{Kutzschebauch-eigentliche-Wirkungen}).
%This complexification is unique as a germ, i.e.
%two complexifications are biholomorphic after shrinking.
In this section the form is extended to a complexification 
$M^c$ of $M$.
This is done equivariantly for groups acting on $M$ and leaving $\eta$ invariant.

\bigskip

\subsection{Equivariant extensions in the case of compact groups}

\bigskip

Let $K$ be a compact transformation group 
and let $M$ be a $K$-manifold 
with a $K$-invariant $1$-form $\eta$.
In the following, an equivariant complexification $M^c$ of $M$
is constructed to which $\eta$ is extended equivariantly.

%\bigskip

\medskip

%\begin{theorem}\label{compact-extension}
\begin{prop}\label{compact-extension}
Let $K$ be a compact Lie group, $M$ a $K$-manifold
and $\eta$ be a $K$-invariant $1$-form.
Then there are a $K$-equivariant complexification $M^c$ of $M$
and a $K$-invariant strictly plurisubharmonic function
$\varrho : M^c \rightarrow \mathbb R$ such that
$$ \iota_M^{\ast}(d^c \varrho) = \eta.$$
%\end{theorem}
\end{prop}

%\medskip

\begin{proof}[{\bf Proof}]
First, the local sitiation without the presence of symmetries is considered.
By a theorem of Whitney (\cite{Whitney-diff-mfd}, Theorem 1) %(p. 654) 
$M$ can be given an atlas with a real analytic structure.
It can also be assumed that the action map $K \times M \rightarrow M$
is real analytic (\cite{Japaner}).
Let $X$ be a complexification of $M$ such that
$ \iota_M : M \hookrightarrow X $
is a real analytic, closed embedding 
(\cite{Whitney-Bruhat}).
It can be assumed that $X$ is a Stein manifold 
(\cite{Grauert}). 
Let $(U_{\alpha}, \varphi_{\alpha})_{\alpha \in I}$ 
be an atlas of real analytic charts
$\varphi_{\alpha}: U_{\alpha} \rightarrow 
                   \varphi_{\alpha}(U_{\alpha})
                             \subset \mathbb R^n$.
Every map
$\varphi_{\alpha}^{-1}: \varphi_{\alpha}(U_{\alpha}) \rightarrow U_{\alpha}$
extends biholomorphically to a map
$ (\varphi_{\alpha}^{-1})^{\mathbb C}: 
  (\varphi_{\alpha})^{\mathbb C}(W_{\alpha}) \rightarrow W_{\alpha}$,
where $W_{\alpha}$ is an open and connected neighbourhood of
$U_{\alpha}$ in $X$ and 
$(\varphi_{\alpha}^{-1})^{\mathbb C}(W_{\alpha})$
is open %and connected in
in $U_{\alpha} \times i \mathbb R^n \subset \mathbb C^n$.
Then $\bigcup_{\alpha} W_{\alpha} $
is an open submanifold of $X$ containing $M$.
After shrinking, this set can be chosen as 
a Stein neighbourhood $M^c$ of $M$ in $X$ (\cite{Grauert}). 
The biholomorphic maps $(\varphi_{\alpha}^{-1})^{\mathbb C}$
have inverse biholomorphic maps, denoted here by
$\varphi_{\alpha}^{\mathbb C}$, which give an atlas 
$(W_{\alpha}, \varphi_{\alpha}^{\mathbb C})$ of $M^{c}$.
Note that
$\varphi_{\alpha}^{\mathbb C}(W_{\alpha})
    \subset U_{\alpha} \times i \mathbb R^n $
is an open neighbourhood of
$\varphi_{\alpha}^{\mathbb C}(W_{\alpha} \cap M) 
    = U_{\alpha} \times \lbrace 0 \rbrace$ 
in $ \mathbb C^n $.
Let $x_1, \ldots, x_n$ be coordinates on 
$\varphi_{\alpha}^{\mathbb C}(W_{\alpha} \cap M) = U_{\alpha} \times \lbrace 0 \rbrace$
and $x_1 + i y_1, \ldots, x_n + i y_n$ coordinates on 
$\varphi_{\alpha}^{\mathbb C}(W_{\alpha})\subset U_{\alpha} \times i \mathbb R^n$.
There are uniquely defined smooth functions 
$f_1, \ldots, f_n : U_{\alpha} \rightarrow \mathbb R$
such that
$$\eta \vert_{U_{\alpha}}(x) = f_1(x_1, \ldots, x_n)dx_1 + 
           \ldots + f_n(x_1, \ldots, x_n)dx_n .$$
On $U_{\alpha} \times i \mathbb R^n \subset \mathbb C^n $
the function
$ \varrho_{\alpha} : U_{\alpha} \times i \mathbb R^n \rightarrow \mathbb R $,
defined by 
$$ \varrho_{\alpha}((x_1 + i y_1, \ldots, x_n + i y_n))
          = f_1(x_1, \ldots, x_n)y_1 + \ldots + f_n(x_1, \ldots, x_n)y_n ,$$
satisfies 
$\varrho_{\alpha} \vert_{U_{\alpha} \times \lbrace 0 \rbrace} \equiv 0$ 
%Since 
%$$ d^c \varrho_{\alpha} \vert_{U_{\alpha}} 
%       = \sum_{j=1}^{n} (y_j d^c f_j + f_j d^c y_j)
%       = \sum_{j=1}^{n} (y_j d^c f_j + f_j d x_j) ,$$
%one also obtains
and this implies that
$$(\iota_{U_{\alpha}})^{\ast}(d^c \varrho_{\alpha})(x_1, \ldots, x_n) 
        = \sum_{j=1}^{n} f_j (x_1, \ldots, x_n) d x_j
        = \eta \vert_{U_{\alpha}} (x_1, \ldots, x_n).$$
Denote by
$ \tilde \varrho_{\alpha} : \varphi_{\alpha}^{\mathbb C}(W_{\alpha})
                              \rightarrow \mathbb R $
the restriction 
$\varrho_{\alpha} \vert_{\varphi_{\alpha}^{\mathbb C}(W_{\alpha})}$.
Then
$\varrho_{\alpha} := (\varphi_{\alpha}^{\mathbb C})^{\ast}(\tilde \varrho_{\alpha})
                   = \tilde \varrho_{\alpha} \circ \varphi_{\alpha}^{\mathbb C}
                 : W_{\alpha} \rightarrow \mathbb R$
has the properties
$  \varrho_{\alpha} \vert_{U_{\alpha} \times \lbrace 0 \rbrace} \equiv 0$
and 
$  (\iota_{U_{\alpha}})^{\ast} (d^c \varrho_{\alpha})
     = \eta \vert_{U_{\alpha}} $
for the embedding 
$\iota_{U_{\alpha}} : U_{\alpha} \hookrightarrow W_{\alpha}$.
%and $\varrho_{\alpha} \vert_{U_{\alpha}} \equiv 0$.
If $p \in U_{\alpha}$,
the sets $U(p) := U_{\alpha}$ and $W(p) := W_{\alpha}$
have the desired properties. %stated in the Proposition.
These locally defined function can be patched together to obtain a function
$\varrho : M^c \rightarrow \mathbb R$
such that $\eta = \iota_M^{\ast} (d^c \varrho)$.
So far, symmetries are not yet considered.
%
%
%Applying Proposition \ref{extension-chart}
There is a Stein complexification $M^c$ of $M$
and an atlas $(W_{\alpha}, \varphi_{\alpha}^{\mathbb C})$
such that 
$(U_{\alpha} := M \cap W_{\alpha}, \varphi_{\alpha}^{\mathbb C} \vert_{U_{\alpha}})$
is an atlas $(U_{\alpha}, \varphi_{\alpha}^{\mathbb C} \vert_{U_{\alpha}})$
and a function
$ \varrho_{\alpha} : W_{\alpha} \rightarrow \mathbb R $
with the properties
\begin{equation}\label{eta-iota}
  \varrho_{\alpha} \vert_{U_{\alpha} \times \lbrace 0 \rbrace} \equiv 0
 \mbox{ and }
  (\iota_{U_{\alpha}})^{\ast} (d^c \varrho_{\alpha})
     = \eta \vert_{U_{\alpha}}
\end{equation}
for the embedding 
$\iota_{U_{\alpha}} : U_{\alpha} \hookrightarrow W_{\alpha}$.
%and $\varrho_{\alpha} \vert_{U_{\alpha}} \equiv 0$.
After shrinking and refining there is an atlas
$(V_{\beta}, \varphi_{\beta})_{\beta \in J}$
of $M^{c}$ with a partition of unity
$(\chi_{\beta})_{\beta \in J}$.
Since every $V_{\beta} \subset W_{\alpha(\beta)}$
for some $\alpha(\beta)$,
it is possible to define
$\varrho_{\beta} := \varrho_{\alpha(\beta)} \vert_{V_{\beta}}$.
Property (\ref{eta-iota}) implies that for
%$M \cap V_{\beta} \not= \emptyset$
the case that $M \cap V_{\beta}$ is non-empty,
%\begin{equation}\label{cut-eta-iota}
%\iota_{M \cap V_{\beta}}^{\ast} (d^c \varrho_{\beta})
%     = \eta \vert_{M \cap V_{\beta}}
%\end{equation}
$ (\iota_{M \cap V_{\beta} \hookrightarrow M^c \cap V_{\beta}})^{\ast} (d^c \varrho_{\beta})
     = \eta \vert_{M \cap V_{\beta}} $
and
\begin{equation}\label{rho=0}
 \varrho_{\beta} \vert_{M \cap V_{\beta}} \equiv 0 .
\end{equation}
Define now for every function $\varrho_{\beta}$
the smooth functions
\[ \begin{array}{rcl}
     \chi_{\beta} \cdot \varrho_{\beta} : 
     M^{c} & \rightarrow & \mathbb R \\
     x & \mapsto &
     \left\{  \begin{array}{rcl}
                (\chi_{\beta} \cdot \varrho_{\beta})(x)
                  & \mathrm{for} & x \in V_{\beta} \\
                0 & \mathrm{for} & x \in M^{c}\setminus V_{\beta}.
             \end{array} \right .
   \end{array} \]
%There are smooth functions on $M^{c}$ and 
The function
%\[ \begin{array}{rcl}
%    \varrho : M^{c} & \rightarrow & \mathbb R \\
%              x & \mapsto &
%              \sum_{\beta} (\chi_{\beta} \cdot \varrho_{\beta})(x)
%   \end{array} \]
$\varrho : M^{c} \rightarrow \mathbb R $,
$ x \mapsto \sum_{\beta} (\chi_{\beta} \cdot \varrho_{\beta})(x) $,
is well-defined by local finiteness and smooth, too.
%Since
%$$ d^c \varrho = \sum_{\beta} \chi_{\beta} \cdot (d^c \varrho_{\beta})
%                           + \varrho_{\beta} \cdot (d^c \chi_{\beta}), $$
Property (\ref{rho=0}) implies
$$ \iota_M^{\ast}(d^c \varrho)
     = \sum_{\beta} \chi_{\beta} \circ \iota_M \cdot
                     \iota_{M}^{\ast} (d^c \varrho_{\beta})
         + 0 \cdot d^c \chi_{\beta} 
     = \sum_{\beta} \chi_{\beta} \circ \iota_M \cdot \eta 
     = \eta. %\mathrm{ \ by \ of \ the \ properties \ of \ a \
             %        partition \ of \ unity}.
$$
Possibly after shrinking $M^c$
%Lemma \ref{Kahler-Lemma} provides
a strictly plurisubharmonic function 
$ \nu : M^{c} \rightarrow \mathbb R $
with the property $\iota_M^{\ast}(d^c \nu) = 0$
%can be added to $\varrho$ 
such that 
$\varrho - \nu$ is strictly plurisubharmonic
on an open neighbourhood of $M$ 
and still satisfies $\iota_M^{\ast}(d^c (\varrho - \nu)) = 0$.
The construction of a function $\nu$ with these properties can be found in 
(\cite{HHL}, Lemma 2). 
Now let $M$ be a $K$-manifold and let $\eta$ be a $K$-invariant $1$-form.
%
%
%There is a real analytic structure on $M$ 
%such that $K \times M \rightarrow M$ is a real analytic action map.
There is a $K$-equivariant complexification $M^c$ of $M$
(\cite{Heinzner-equivariant-extensions}, \cite{Japaner}, Theorem 1.3).
In particular $K$ acts on $M^c$ by holomorphic transformations
and $\iota_M : M \hookrightarrow M^c$ is a $K$-equivariant embedding.
%Note that every neighbourhood of $M$ in $M^c$ contains a $K$-invariant one.
Perhaps after shrinking to a smaller $K$-invariant complexification 
$M^c$ has a smooth strictly plurisubharmonic function 
$\tilde \varrho : M^c \rightarrow \mathbb R$ such that
$ \iota_M^{\ast}(d^c \tilde \varrho) = \eta $
%by Proposition \ref{compl-dc-rho-eta}.
as shown above.
Then $\varrho(x) := \int_K \tilde \varrho(\orbit{k^{-1}}{x}) dk$
defines a $K$-invariant strictly plurisubharmonic function on $M^c$ such that
\begin{align*}
& \iota_M^{\ast}(d^c \varrho) 
      = \iota_M^{\ast} d^c (\int_K \tilde \varrho \circ \psi_{k^{-1}} dk)
      = \iota_M^{\ast} (\int_K \psi_{k^{-1}}^{\ast} (d^c \tilde \varrho) dk) \\
& \phantom{\iota_M^{\ast}(d^c \varrho)}  
      = \int_K \psi_{k^{-1}}^{\ast}(\iota_M^{\ast} (d^c \tilde \varrho)) dk 
      = \int_K \eta dk
      = \eta, 
\end{align*}
%where $\psi : K \times M^c \rightarrow M^c$ is the $K$-action on $M^c$.
where $\psi_k : M^c \rightarrow M^c$, $\psi_k(x) = \orbit{k}{x}$, for all $k \in K$.
\end{proof}

%\bigskip

%\section{Equivariant extensions in the case of proper group actions}\label{Equivariant extensions in the case of proper group actions}

%\bigskip

%An extendable Lie group is characterized by 
%the injectivity of the canonical $G$-equivariant homomorphism 
%$\iota_G : G \rightarrow G^{\mathbb C}$, 
%where $G^{\mathbb C}$ is the universal complexification.
%The aim of the second section is to generalize 
%the results of the previous section to proper actions 
%of extentable
%%\footnote{An extendable Lie group is characterized by the injectivity of the canonical $G$-equivariant homomorphism $\iota_G : G \rightarrow G^{\mathbb C}$, where $G^{\mathbb C}$ is the universal complexification.}  
%Lie groups with finitely many connected components.
%In this setting
%Theorem \ref{Komplexifizierungssatz}
%shows that a $G$-invariant $1$-form $\eta$
%can be extended equivariantly to a complexification
%$\iota_M : M \hookrightarrow M^c$ as 
%$ \iota_M^{\ast} (d^c \varrho) = \eta$,
%where $\varrho : M^c \rightarrow \mathbb R$
%is $G$-invariant and strictly plurisubharmonic.

\bigskip

%\subsection{Slice properties and equivariant complexifications of proper actions}

\subsection{Equivariant extensions for the case of proper actions}

%\subsection{Extensions of invariant forms for proper actions}\label{Extensions of invariant forms for proper actions}

\bigskip

An extendable Lie group is characterized by 
the injectivity of the canonical $G$-equivariant homomorphism 
$\iota_G : G \rightarrow G^{\mathbb C}$, 
where $G^{\mathbb C}$ is the universal complexification.
The aim of this subsection is the following result.

\medskip

\begin{theorem}\label{Komplexifizierungssatz}
Let $G$ be an extendable Lie group
with finitely many connected components
that acts properly on a manifold $M$
and let $K$ be a maximal compact subgroup of $G$.
Let $\eta$ be a smooth $G$-invariant $1$-form on $M$.
The slice $S \subset M$ is embedded in a Stein 
$K^{\mathbb C}$-manifold $S^{\mathbb C}$
such that $M = G \times^K S$ is complexified by a 
$G \times K$-invariant Stein domain 
$M^c \subset G^{\mathbb C} \times^{K^{\mathbb C}} S^{\mathbb C}$.
Then there is a $G$-invariant strictly plurisubharmonic function 
$\varrho: M^c \rightarrow \mathbb R $
such that 
$$ \iota_M^{\ast} d^c \varrho = \eta.$$
\end{theorem}

\bigskip

The proof of this at the end of the section needs some preperation.
Let $G$ be an extendable 
%\footnote{An extendable Lie group is characterized by the injectivity of the canonical $G$-equivariant homomorphism $\iota_G : G \rightarrow G^{\mathbb C}$, where $G^{\mathbb C}$ is the universal complexification.}
Lie group with finitely many connected components 
and let $K$ be a maximal compact subgroup of $G$.
By a theorem of Abels (\cite{Abels})
there is a $K$-invariant submanifold $S$ in $M$
such that the map
\[ \begin{array}{rcl}   
        G \times^{K} S & \rightarrow & M \\
        \lbrack g, s \rbrack & \mapsto & g \cdot s
      \end{array} \] 
is a diffeomorphism.
Here, $G \times^K S$ denotes the geometric quotient of $G \times S$
with respect to the free $K$-action
\[ \begin{array}{rcl}
    K \times (G \times S) & \rightarrow & G \times S \\
    (k, (g,s)) & \mapsto & (gk^{-1}, \orbit{k}{s})
   \end{array} \]
and $\lbrack g, s \rbrack := \pi_K(g, s)$,
where $\pi_K : G \times S \rightarrow G \times^K S $ 
is the canonical projection onto their geometric quotient.
%By a result of (\cite{Kutzschebauch-uniqueness}, \cite{Illman}), %\cite{Japaner},
There is a real analytic structure on 
$M \cong G \times^K S$ such that 
the action map $G \times M \rightarrow M$,
the slice $S$ and the $K$-action on $S$ 
may be assumed to be real analytic (\cite{Illman}, \cite{Kutzschebauch-uniqueness}).
%
%In \cite{HHK}\footnote{\cite{HHK}, Section 7, p. 264, Proposition 4, Proposition 4' and Proposition 5.}
In (\cite{HHK}, Section 7, Proposition 4, 4' and 5)
a complexification of $G \times^K S$
is constructed with the help of
a $G$-complexification $G^c$ of $G$
and a $K$-complexification $S^c$ of $S$ 
as the quotient $\ktquot{G^c \times S^c}{K}$.
%is a $G$-complexification of $G \times^K S$, where $K$ acts diagonally.
In this quotient two points $p_1$ and $p_2$
are identified if $f(p_1) = f(p_2)$
for every $K$-invariant holomorphic function $f$.
%Since $G$ is assumed to be extendable here,
%a modification of the approach in \cite{HHK}
%proves that a $G$-complexification
%can be realized as a $G^{\mathbb C}$-manifold:
Since $G$ is assumed to be extendable here,
a $G$-complexification can be realized as a $G^{\mathbb C}$-manifold
as in \cite{HHK}. The proof is included here for the readers' convenience.

\bigskip

\begin{prop}\label{KC-mfd}
Let an extendable Lie group $G$ act properly 
and real analy\-tically on a manifold $M = G \times^K S$,
where $K$ is a maximal compact subgroup of $G$.
Then there is a $K^{\mathbb C}$-manifold $S^{\mathbb C}$
%in which a $K$-complexification $S^c$ is embedded 
such that a $G$-invariant domain $\Omega$ 
in $M^{\mathbb C} = G^{\mathbb C} \times^{K^{\mathbb C}} S^{\mathbb C}$
is a $G$-complexification of $M = G \times^K S$.
\end{prop}

%The complexification results can be summarized in the following Proposition.

%\begin{prop}\label{embedding-in G-C-mfd}
%Let the extendable Lie group $G$ 
%with finitely many components
%act properly on a manifold $M$
%and let $M^c$ be an equivariant $G$-complexification.
%Then a $G$-invariant, open neighbourhood $\Omega$ of $M$ in $M^c$
%can be openly, $G$-equivariantly and holomorphically
%embedded in a $G^{\mathbb C}$-manifold $M^{\mathbb C}$.
%\end{prop}

%\bigskip

\begin{proof}[{\bf Proof}]
The slice $S$ can be $K$-equivariantly complexified
in a Stein $K^{\mathbb C}$-space $S^{\mathbb C}$
(\cite{Heinzner-geometric-invariant-stein}, Section 6.6).
Since $G$ is extendable, 
it can be complexified $G$-equivariantly 
to a $G$-invariant open domain $G^c$ in $G^{\mathbb C}$.
Then $M \cong G \times^K S$ can be $G$-equivariantly embedded in 
$M^{\mathbb C} := G^{\mathbb C} \times^{K^{\mathbb C}} S^{\mathbb C}$
as a totally real submanifold.
If $M^c$ is a $G$-complexification of $M$, 
a $G$-invariant domain $\Omega$ containing $M$ in $M^c$
can be $G$-equivariantly, holomorphically and openly embedded 
in a neighbourhood of $M$ in $M^{\mathbb C}$
(\cite{HHK}, Corollary 7). %(p. 266).
\end{proof}

\bigskip

Following the notation introduced above
consider a proper and real analytic action on $M$. 
Then $M = G \times^K S$.
Let $\eta$ be a $G$-invariant smooth $1$-form on $M$.
Denote by $\pi_G : G \times S \rightarrow G$
and $\pi_S : G \times S \rightarrow S$
the projections on the first and on the second factor 
respectively.

\bigskip

\begin{prop}\label{prop-frame}
Let $\eta$ be a $G$-invariant, smooth $1$-form on $G \times^K S$.
Let $\beta_1, \ldots, \beta_n$
be a basis of $G$-invariant $1$-forms on $G$.
Then there are smooth functions 
$f_1, \ldots, f_n: S \rightarrow \mathbb R$
and a $K$-invariant $1$-form $\sigma_S$ on $S$ such that 
$$ \pi_K^{\ast}\eta  
        %  = f_1 \beta_1 + \ldots + f_n \beta_n + \sigma_S 
     = \sum_{j=1}^{n} \pi_S^{\ast}(f_j) \cdot \pi_G^{\ast}(\beta_j) 
        +  \pi_S^{\ast}(\sigma_S) $$ 
such that $\sum_{j=1}^{n} \pi_S^{\ast}(f_j) \cdot \pi_G^{\ast}(\beta_j)$ 
is a $G \times K$-invariant $1$-form.
\end{prop}

%\bigskip

\begin{proof}[{\bf Proof}]
The form $\pi_K^{\ast} \eta$ is a $G \times K$-invariant, 
smooth $1$-form on $G \times S$.
Let $\beta_1, \ldots, \beta_n$ be a basis of $G$-invariant $1$-forms on $G$;
then $\pi_G^{\ast}(\beta_1), \ldots, \pi_G^{\ast}(\beta_n)$
are their trivial extensions to $G \times S$.
The embedding 
$\iota_S : S \rightarrow G \times S$, $s \mapsto (e,s)$,
and the projection
$\pi_S : G \times S \rightarrow S$, $(g,s) \mapsto s$,
are $K$-equivariant if $K$ acts diagonally on $G \times S$ by 
$\orbit{k}{(g,s)} = (g k^{-1}, \orbit{k}{s})$.
Let $\sigma_S = \iota_S^{\ast}(\pi_K^{\ast} \eta)$;
then $\pi_S^{\ast} (\sigma_S)$ is $G \times K$-invariant 
and for every tangent vector 
$(0,v) \in T_{(g,s)}(G \times S) \cong T_g G \times T_s S $ 
$$(\pi_S^{\ast}(\sigma_S))(0,v) 
   = \sigma_S (D \pi_S (0,v)) 
   = (\pi_K^{\ast} \eta) (D \iota_S (D \pi_S (0,v))) 
   = (\pi_K^{\ast} \eta) (0,v). $$
It follows that
$(\pi_K^{\ast} \eta - \pi_S^{\ast}(\sigma_S)) (h,s) (0,v) = 0 $
for every $v \in T_s S$ and every $h \in G$.
In other words
$ (\pi_K^{\ast} \eta - \pi_S^{\ast}(\sigma_S)) (h,s) 
    \in T_h^{\ast}G \oplus \lbrace 0 \rbrace 
    \subset T_h^{\ast}G \oplus T_s^{\ast}S \cong T_{(h,s)} G \times S $.
This implies that there are smooth functions
$ \tilde f_1, \ldots, \tilde f_n$ % $ \in \mathcal{C}^{\infty}(G \times S)$
on $G \times S$ such that 
$$ (\pi_K^{\ast} \eta - \pi_S^{\ast}(\sigma_S)) (h,s) 
   = \sum_{j=1}^{n} \tilde f_j(h,s) \cdot \pi_G^{\ast}(\beta_j) .  $$
%Since the action map 
%$\psi_g : G \rightarrow G$, $h \mapsto gh$,
%has the invariance properties
%$$ (\psi_g)^{\ast} (\pi_K^{\ast} \eta) = \pi_K^{\ast} \eta 
%      \mbox{ and } 
%   (\psi_g)^{\ast} (\pi_S^{\ast}(\sigma_S)) = \pi_S^{\ast}(\sigma), $$
%the identity
%\begin{align*}
%  \sum_{j = 1}^{n} \tilde f_j (h,s) \cdot \pi_G^{\ast}(\beta_j)(h,s)
% & = (\pi_K^{\ast} \eta - \pi_S^{\ast}(\sigma_S)) (h,s) \\
%% & = (\psi_g)^{\ast} (\pi_K^{\ast} \eta - \pi_S^{\ast}(\sigma_S)) (h,s) \\
%% & = \sum_{j = 1}^{n} (\tilde f_j \circ \psi_g) (h,s) \cdot 
%%                       (\psi_g)^{\ast}(\pi_G^{\ast}(\beta_j))(h,s) \\
% & = \sum_{j = 1}^{n} \tilde f_j (gh,s) \cdot \pi_G^{\ast}(\beta_j)(h,s)
%\end{align*}
%%$$ \sum_{j = 1}^{n} \tilde f_j (h,s) \cdot \pi_G^{\ast}(\beta_j)(h,s)
%%  = (\pi_K^{\ast} \eta - \pi_S^{\ast}(\sigma_S)) (h,s) 
%%  = \sum_{j = 1}^{n} \tilde f_j (gh,s) \cdot \pi_G^{\ast}(\beta_j)(h,s) $$
%holds, where $\iota_S^{\ast}(\pi_K^{\ast} \eta) = \sigma_S$ as above.
Comparing coefficients implies that 
there are smooth functions $f_1,  \ldots, f_n$ on $S$ such that 
$$ (\pi_K^{\ast} \eta - \pi_S^{\ast}(\sigma_S)) (h,s) 
   = \sum_{j=1}^{n} \pi_S^{\ast}(f_j)(s) \cdot \pi_G^{\ast}(\beta_j). $$
\end{proof}

%\bigskip

%\subsection{Extensions of invariant forms in the case of free and proper actions}\label{Extensions of invariant forms in the case of free and proper actions}

\bigskip

%Proposition \ref{compl-dc-rho-eta}
Proposition \ref{compact-extension}
implies that 
there is a $K$-invariant strictly plurisubharmonic function
$ \varrho_S : S^c \rightarrow \mathbb{R} $
on an equivariant $K$-complexification $S^c$ of $S$
such that 
$\sigma_S = (\iota_{S \hookrightarrow S^c})^{\ast} (d^c \varrho_S)$. 
 %$  = \sigma_S $. % = h_1(s)\sigma_1 + \ldots + h_k(s)\sigma_k $.
%
%
Assume that the situation is arranged as in Proposition \ref{KC-mfd}.
Let $S^c$ be openly and $K$-equivariantly embedded 
in a $K^{\mathbb C}$-manifold $S^{\mathbb C}$
and let $G^c$ be a Stein $G$-complexification of $G$
which is $G$-equivariantly and openly embedded in $G^{\mathbb C}$.
In the next proposition, $G \times K$-invariant $1$-forms
%$\sum_{j=1}^{n} \pi_S^{\ast}(f_j) \cdot \pi_G^{\ast}(\beta_j)$
on $G \times S$ are going to be extended equivariantly to $G^c \times S^c$.

\bigskip

\begin{prop}\label{G-times-K-invariante-function}
For the smooth $G \times K$-invariant $1$-form
$\pi_K^{\ast}\eta$ on $G \times S$
there is a $G \times K$-invariant strictly plurisubharmonic function
$\varrho$ on some $G \times K$-invariant complexification $G^c \times S^c$ 
such that on a $G \times K$-invariant Stein domain $\Omega$
in $G^c \times S^c$ 
$$ (\iota_{G \times S})^{\ast} (d^c \varrho) = \pi_K^{\ast}\eta.$$
\end{prop}

\smallskip

\begin{proof}[{\bf Proof}]
%\begin{prop}\label{prop-Theta-basis}
Let $\beta_1, \ldots, \beta_n$
be a basis of $G$-invariant 
$1$-forms on $G$ and 
$f_1, \ldots, f_n \in \mathcal{C}^{\infty}(S)$ 
and 
$ \pi_K^{\ast}\eta = \sum_{j=1}^{n} \pi_S^{\ast}(f_j) \cdot \pi_G^{\ast}(\beta_j)  
                      + \pi_S^{\ast}(\sigma_S)$
be as above in Proposition \ref{prop-frame}. 
%Then there is 
%%a set $G^c \subset G^{\mathbb C}$ which is 
%%invariant under the left multiplication of $G$ on $G^c$
%a $G$-equivariant Stein complexification $G^c$ of $G$ 
%and a Stein complexification $S^c$ of $S$
%and a $G$-invariant function
%$\Theta : G^c \times S^c \rightarrow \mathbb R$ 
%%on a complexification $G^c \times S^c$ of $G \times S$
%such that
%$$ (\iota_{G \times S})^{\ast} (d^c \Theta) 
%      % = f_1 (s) \beta_1 + \ldots + f_n (s) \beta_n, 
%   =\sum_{j=1}^{n} \pi_S^{\ast}(f_j) \cdot \pi_G^{\ast}(\beta_j),$$
%where $\iota_{G \times S} : G \times S \hookrightarrow G^c \times S^c$ 
%is the embedding. 
%%and such that $\Theta$ is invariant under the left $G$-action on $G^c$.
%\end{prop}
%
%
%\begin{proof}[{\bf Proof}]
Let $G^c$ be a Stein complexification of $G$
which is $G$-equivariant with respect to the left $G$-multiplication.
Shrinking $G^c$ if necessary, 
%Lemma \ref{lem-4-3-strat}
Lemma 3.3 in \cite{Stratmann} and Theorem 1 in \cite{Winkelmann}
imply that
there are $G$-invariant functions 
$\varrho_1, \ldots, \varrho_n$ 
%on %$G$-tubes 
%certain $G$-invariant Stein neighbourhoods 
%$G^{c}_1, \ldots, G^{c}_n$ of $G$ in %$G^{\mathbb C}$
%$G^c$, which one obtains by shrinking $G^{\mathbb C}$
%on a possibly shrinked $G^c$
on $G^c$
such that $ \iota_G^{\ast} (d^c \varrho_j) = \beta_j$
for $j = 1, \ldots, n$. 
%on $G^c_k$
%for every $k = 1, \ldots, n$.\footnote{Kutzschebauch (\cite{Kutzschebauch-eigentliche-Wirkungen}, \S 2.1, Korollar zu Lemma 3) shows that for a compact subgroup $K$ of an extendable group $G$, there is a neighbourhood basis of $G \times K$-invariant Stein and $H$-orbit-convex neighbourhoods of $G$ in $G^c$. This is why, for later use, the subsets $G_1^c, \ldots, G_n^c$ of $G^c$ can be assumed $G \times K$-invariant, Stein and $H$-orbit-convex.}
It can be assumed that %every  
$\varrho_j \vert_{G} \equiv 0$.
%Replace $G^c$ by $G^c := G^{c}_1 \cap \ldots \cap G^{c}_n$. 
There is a complexification $S^c$ of $S$ with functions 
$F_1, \ldots, F_n :S^c \rightarrow \mathbb R$
%[{\bf Invarianz kommt erst sp\"ater.}]
such that 
$ f_j = (\iota_{S \hookrightarrow S^c})^{\ast} (F_j)$
for $j = 1, \ldots, n$.
%Then define the function
The $G$-invariant function
\[ \begin{array}{rcl}   
    \Theta : G^c \times S^c & \rightarrow & \mathbb R \\ 
        (h, s) & \mapsto & 
          F_1(s) \varrho_1(h) + \ldots + F_n(s) \varrho_n(h)
        \end{array} \] 
%which is invariant under the $G$-action 
%$G \times ( G^c \times S^c ) \rightarrow G^c \times S^c$, 
%$(g, (h, s)) \mapsto (gh, s)$.
%Since 
satisfies
%$ d^c \Theta   % =  F_1(s) d^c \varrho_1(h) + \varrho_1(h) d^c F_1(s) 
                % + \ldots + 
                % F_n(s) d^c \varrho_n(h) + \varrho_n(h) d^c F_n(s) 
%              = \sum_{j=1}^{n} F_j(s) d^c \varrho_j(h) + \varrho_j(h) d^c F_j(s)$.
$ d^c \Theta 
  = \sum_{j=1}^{n} \pi_{S^c}^{\ast}(F_j) \cdot \pi_{G^c}^{\ast}(d^c \varrho_j) 
                 + \pi_{G^c}^{\ast}(\varrho_j) \cdot \pi_{S^c}^{\ast}(d^c F_j)$
with the projections
$\pi_{G^c}: G^c \times S^c \rightarrow G^c$ and
$\pi_{S^c}: G^c \times S^c \rightarrow S^c$.
Now, the property 
$\varrho_j \vert_{G} \equiv 0$ %for $j = 1, \ldots n$, 
implies
%\begin{align*}
%  \iota_{G \times S}^{\ast} (d^c \Theta)
%    & = f_1(s) \beta_1 + 0 \cdot \iota_{G \times S}^{\ast} d^c F_1 
%                + \ldots + 
%        f_n(s) \beta_n + 0 \cdot \iota_{G \times S}^{\ast} d^c F_n \\
%    & = f_1(s) \beta_1 + \ldots + f_n(s) \beta_n.  
%\end{align*}
$$ (\iota_{G \times S})^{\ast} (d^c \Theta)
   = \sum_{j=1}^{n} \pi_S^{\ast}(f_j) \cdot \pi_G^{\ast}(\beta_j).  $$
%
%\begin{proof}[{\bf Proof}]
% Let $\beta_1, \ldots, \beta_n$
% be a basis of $G$-invariant $1$-forms on $G$.
% Then, %by $G$-invariance, 
% by Proposition \ref{prop-frame}, 
%
%By Lemma \ref{lem-4-3-strat} there is 
%a $G$-equivariant complexification $G^c$ of $G$,
%which can be made $G \times K$-invariant
%by integration over $K$.
%By Corollary \ref{compact-extension},
%a $K$-equivariant complexification $S^c$ of $S$
%is provided.
%There are a $K$-invariant $1$-form $\sigma_S$ on $S$ 
%and $K$-invariant functions $f_1, \ldots, f_k \in \mathcal{C}^{\infty}(S)$ 
%such that 
%%$$ \pi_K^{\ast}\eta = f_1(s)\beta_1 + \ldots + f_n(s) \beta_n + \pi_S^{\ast}(\sigma_S),$$
%
%
%Let
%$ \pi_K^{\ast}\eta = \sum_{j=1}^{n} \pi_S^{\ast}(f_j) \cdot \pi_G^{\ast}(\beta_j)  
%                      + \pi_S^{\ast}(\sigma_S)$
%be as above in Proposition \ref{prop-frame}.
%%where $\pi_S : G \times S \rightarrow S$
%%is the projection onto the second factor.
%By Proposition %%\ref{lem-4-3-strat} 
%\ref{prop-Theta-basis}
%there is a $G \times K$-equivariant complexification $G^c \times S^c$ 
%and a $G$-invariant function $\Theta : G^c \times S^c \rightarrow \mathbb R$
%such that 
%$$ (\iota_{G \times S})^{\ast} (d^c \Theta) 
%      % = f_1(s)\beta_1 + \ldots + f_n(s) \beta_n 
%   = \sum_{j=1}^{n} \pi_S^{\ast}(f_j) \cdot \pi_G^{\ast}(\beta_j)  .$$
After shrinking $S^c$
%Proposition \ref{compl-dc-rho-eta}
Proposition \ref{compact-extension} shows that 
there is a strictly plurisubharmonic $K$-invariant function
$ \theta : S^c \rightarrow \mathbb R $
such that 
$(\iota_{S \hookrightarrow S^c})^{\ast} (d^c \theta) 
    = \sigma_S $. 
Define the $G \times K$-invariant function 
$\varrho$ by 
$$ \varrho (g, s) := %\frac{1}{\mathrm{vol}(K)} 
    \int_K ( \Theta (g k^{-1}, k s) + (\pi_{S^c})^{\ast}(\theta) (k s) ) dk.$$
Then $(\iota_{G \times S})^{\ast}(d^c \varrho) = \pi_K^{\ast}\eta$, 
and $\varrho$ is a $G \times K$-invariant function.
%\end{proof}
A partition of unity argument, worked out in Lemma 3.10 in \cite{Stratmann},
which is e.g. shows that $\varrho$ can be assumed both
$G \times K$-invariant and strictly plurisubharmonic.
\end{proof}

%\bigskip

%\subsection{Extensions of invariant forms for proper actions}\label{Extensions of invariant forms for proper actions}

\bigskip

Recall the original goal to extend a $G$-invariant $1$-form $\eta$
to an equivariant complexification of $G \times^K S$.
This will be achieved with the help of K\"ahlerian reduction of $G^c \times S^c$
with respect to the freely acting compact group $K$.
Details on the momentum map geometry and on K\"ahlerian reduction 
can be found e.g. in  
%\cite{Ammon},
%\cite{Guillemin-Sternberg},
\cite{Heinzner-geometric-invariant-stein}, 
%%\cite{HH-invent},
%%\cite{HH-ana-hilbert-quot},
%\cite{HH-Schneider-Band},
\cite{Heinzner-Loose}, 
\cite{Sjamaar}. 
The basic properties needed here
are mentioned briefly in the remaining section.

\medskip

Let $(\Omega, \omega)$ be a K\"ahler manifold 
and let $L$ be a Lie group 
which acts symplectically 
and by holomorphic transformations
on $\Omega$, 
i.e., $(\psi_g)^{\ast} \omega = \omega$ for every $g \in L$,
where $\psi : L \times \Omega \rightarrow \Omega$ is the action map.
The action is called Hamiltonian 
if there is a moment map $\mu : \Omega \rightarrow \mathrm{Lie}(L)^{\ast}$, 
where $\mathrm{Lie}(L)^{\ast}$ is the dual vector space to the Lie algebra of $L$,
with the following properties:
\begin{enumerate}
\item The map $\mu$ is $L$-equivariant 
      with respect to the given action on $\Omega$ and the coadjoint action of 
      $L$ on $\mathrm{Lie}(L)^{\ast}$.
\item For every $\xi \in \mathrm{Lie}(L)$
      the function
      $ \mu_{\xi} : \Omega \rightarrow \mathbb R$, 
      $x \mapsto \langle \mu(x), \xi_{\Omega}(x) \rangle$,
      satisfies %the identity of $1$-forms
      $\iota_{\xi_{\Omega}} \omega = d\mu_{\xi} $,
      where $\xi_{\Omega}(x) = \frac{d}{dt} \orbit{\exp(t \xi)}{x} \big \vert_{t=0}$
      and $\iota_{\xi_{\Omega}} \omega$ 
      is the $1$-form given by 
      $(\iota_{\xi_{\Omega}} \omega)(v) = \omega(\xi_{\Omega}, v)$ 
      for every $v \in TM$.
\end{enumerate}
If $\Omega$ carries a differentiable, 
$L$-invariant strictly plurisubharmonic function
$\varrho : \Omega \rightarrow \mathbb R$,
the action is Hamiltonian
with respect to the K\"ahler metric $\omega = - dd^c \varrho$.
In this case, a moment map is given by
$\mu_{\xi}(x) = (d^c \varrho)(\xi_{\Omega}(x)) .$
For Hamiltonian actions,  
the momentum zero level 
$ \mu^{-1}(0) = \lbrace x \in \Omega \vert 
                      \mu_{\xi}(x) = 0 \mbox{ for all } \xi \in \mathrm{Lie}(L) 
                 \rbrace $
allows one to define the reduced space 
$\modulo{\mu^{-1}(0)}{L}$.

\bigskip

\begin{prop}\label{free-proper-case}
Let $L$ act freely and properly by holomorphic and symplectic transformations
%on an $L$-invariant domain $\Omega$
%with an $L$-invariant K\"ahler form $\omega$
on a K\"ahler manifold $(\Omega, \omega)$.
Assume that the action is Hamiltonian
with moment map $\mu : \Omega \rightarrow \mathrm{Lie}(L)^{\ast}$.
Let $\Omega \hookrightarrow X$ 
be openly, holomorphically and $L$-equivariantly embedded 
in a complex $L^{\mathbb C}$-manifold $X$
on which $L^{\mathbb C}$ acts freely
such that $\modulo{X}{L^{\mathbb C}}$
is a smooth complex manifold and
$\pi_{L^{\mathbb C}} : X \rightarrow \modulo{X}{L^{\mathbb C}}$
a submersion.
Then the map
\[ \begin{array}{rcl}
 \kappa : \modulo{\mu^{-1}(0)}{L} & \rightarrow & \modulo{X}{L^{\mathbb C}} \\
            L x_0 & \mapsto & \pi_{L^{\mathbb C}}(\iota_{\mu^{-1}(0) \hookrightarrow X} (x_0)) 
  \end{array} \]
is a local diffeomorphism
and defines a unique complex structure on $\modulo{\mu^{-1}(0)}{L}$
such that $\kappa$ is a locally biholomorphic map of complex manifolds.
\end{prop}

%\bigskip

\smallskip

\begin{proof}[{\bf Proof}]
Since 
$\ker d\mu(x) 
   = (T_{x} \orbit{L}{x})^{\bot_{\omega}} 
   = \lbrace v \in T_{x}X \vert \omega(v,w) = 0 \ \forall w \in T_{x} \orbit{L}{x} \rbrace$, 
$x \in X$,
implies that
$\mathrm{rank}(\mu) = \dim (\mathrm{Lie}(L))$ everywhere, 
$\mu^{-1}(0)$ is a smooth submanifold of $\Omega$.
The fact that $L$ acts freely and properly
implies that $\modulo{\mu^{-1}(0)}{L}$
is likewise a differentiable manifold.
Furthermore, for a point $x_0 \in \mu^{-1}(0)$,
$$ T_{x_0} \mu^{-1}(0) = T_{x_0} (\orbit{L}{x_0}) \oplus 
                         T_{x_0} (\orbit{L^{\mathbb C}}{x_0})^{\bot_{\omega}}.$$
In the commutative diagram
\[ \begin{array}{rccl}
    \mu^{-1}(0) \phantom{xx} & \stackrel{\iota_{\mu^{-1}(0)}}{\longhookrightarrow} & X \phantom{xx} \\
    & & \\
    \phantom{xxx} \Big \downarrow  \mbox{ }^{ \pi_{\mu^{-1}(0)} } & & \Big \downarrow \mbox{ }^{ \pi_{L^{\mathbb C}} } \\
    & & \\
    \modulo{\mu^{-1}(0)}{L} \phantom{x} & \stackrel{\kappa}{\longrightarrow} &  
    \modulo{X}{L^{\mathbb C}}, 
   \end{array} \]
the maps 
$\pi_{\mu^{-1}(0)} : \mu^{-1}(0) \rightarrow \modulo{\mu^{-1}(0)}{L}$
and 
$\pi_{L^{\mathbb C}} : X \rightarrow \modulo{X}{L^{\mathbb C}}$
are submersions with kernels
$\ker(D \pi_{\mu^{-1}(0)})(x) = T_x(\orbit{L}{x}) $ and 
$\ker(D \pi_{L^{\mathbb C}})(x) = T_x(\orbit{L^{\mathbb C}}{x})$
respectively.
It follows that
$D(\pi_{L^{\mathbb C}} \circ \iota_{\mu^{-1}(0)})(x_0) = D(\kappa \circ \pi_{\mu^{-1}(0)})(x_0)$
maps $T_{x_0} (\orbit{L^{\mathbb C}}{x_0})^{\bot_{\omega}}$
bijectively onto 
$T_{\pi_{L^{\mathbb C}}(x_0)} \Big( \modulo{X}{L^{\mathbb C}} \Big)$
and $D \pi_{\mu^{-1}(0)}$ maps
bijectively onto 
$T_{\pi_{L^{\mathbb C}}(x_0)} \Big( \modulo{\mu^{-1}(0)}{L} \Big) $.
This implies that $D \kappa (\pi_{\mu^{-1}(0)}(x_0))$
is everywhere an isomorphism.
\end{proof}

\bigskip

Let $\Omega \subset G^c \times S^c$ 
be a $G \times K$-invariant Stein domain,
$G \times S \subset \Omega$,  
and $\varrho: \Omega \rightarrow \mathbb R$
a $G \times K$-invariant strictly plurisubharmonic function
such that 
$(\iota_{G \times S})^{\ast}(d^c \varrho) = \pi_K^{\ast} \eta$.

\bigskip

\begin{prop}\label{Kempf-Ness-free-case}
There are $G$-invariant Stein domains
$\Omega_1 \subset \modulo{\mu^{-1}(0)}{K}$ 
containing $G \times^K S$ and 
$\Omega_2 \subset G^{\mathbb C} \times^{K^{\mathbb C}} S^{\mathbb C}$ 
containing $G \times^K S$ 
which are $G$-equivariantly biholomorphic.
\end{prop}

%\medskip

\begin{proof}[{\bf Proof}]
First, it has to be shown that
$G \times S \subset \mu^{-1}(0)$.
For this, %it remains to prove that 
%a short 
the following
calculation proves that
for every $\zeta \in \mathrm{Lie}(K)$ 
and every $(g, s) \in G \times S$,
%$$ \frac{d}{dt} \varrho(\orbit{\exp(it \zeta)}{(g, s)}) \big \vert_{t=0} = 0.$$
%This follows because
\begin{align*}
 \frac{d}{dt} \varrho(\orbit{\exp(it \zeta)}{(g, s)}) \big \vert_{t=0}  
& = (\iota_{G \times S})^{\ast} (d^c \varrho) 
    \Big( \frac{d}{dt} (g \exp(-t \zeta), \orbit{\exp(t \zeta)}{s} ) \big \vert_{t=0} \Big) \\
& = (\pi_K^{\ast} \eta) 
    \Big( \frac{d}{dt} (g \exp(-t \zeta), \orbit{\exp(t \zeta)}{s} ) \big \vert_{t=0} \Big) \\
& = \eta \Big( \frac{d}{dt} \pi_K(g, s) \big \vert_{t=0} \Big) \\ 
& = 0.
\end{align*}
Since $K$ acts freely and commutes with the $G$-action,
$\modulo{\mu^{-1}(0)}{K}$is a $G$-manifold and,
by Proposition \ref{free-proper-case},
obtains a complex structure by the map
$$ \kappa : \modulo{\mu^{-1}(0)}{K} \rightarrow 
          G^{\mathbb C} \times^{K^{\mathbb C}} S^{\mathbb C}, $$
which is a local diffeomorphism. 
Since the restriction $\kappa \vert_{G \times^K S}$
defines a real analytic, $G$-equivariant isomorphism between
two copies of 
$G \times^K S$ in $\modulo{\mu^{-1}(0)}{K}$
and in 
$ G^{\mathbb C} \times^{K^{\mathbb C}} S^{\mathbb C} $ respectively,
there are $G$-invariant and biholomorphic Stein neighbourhoods 
$\Omega_1$ of $G \times^K S$ in $\modulo{\mu^{-1}(0)}{K}$
and
$\Omega_2$ of $G \times^K S$ in 
$G^{\mathbb C} \times^{K^{\mathbb C}} S^{\mathbb C}$
(\cite{HHK}, Corollary 7). %(p. 266).
\end{proof}

\bigskip

It remains to show how to use the extension of $\pi_K^{\ast} \eta$ on 
$G^c \times S^c$ for an extension of $\eta$ on $G \times^K S$. 
%to the quotient $\ktquot{G^c \times S^c}{K}$.
This is carried out in the proof of the %following 
Theorem \ref{Komplexifizierungssatz}, 
which can now be carried out.

\bigskip

%\begin{theorem}\label{Komplexifizierungssatz}
%Let $G$ be an extendable Lie group
%with finitely many connected components
%that acts properly on a manifold $M$
%and let $K$ be a maximal compact subgroup of $G$.
%Let $\eta$ be a smooth $G$-invariant $1$-form on $M$.
%The slice $S \subset M$ is embedded in a Stein 
%$K^{\mathbb C}$-manifold $S^{\mathbb C}$
%such that $M = G \times^K S$ is complexified by a 
%$G \times K$-invariant Stein domain 
%$M^c \subset G^{\mathbb C} \times^{K^{\mathbb C}} S^{\mathbb C}$.
%Then there is a $G$-invariant strictly plurisubharmonic function 
%$\varrho: M^c \rightarrow \mathbb R $
%such that 
%$$ \iota_M^{\ast} d^c \varrho = \eta.$$
%\end{theorem}

%\medskip

\begin{proof}[{\bf Proof of Theorem \ref{Komplexifizierungssatz}}]
By Proposition \ref{G-times-K-invariante-function}
%and Lemma \ref{Equivariant-Kahler-Lemma}
%and by Lemma 3.10 in \cite{Stratmann},
there is a Stein $G \times K$-complexifica\-tion 
$\Omega \subset G^{\mathbb C} \times S^{\mathbb C}$ of $G \times S$
and a strictly plurisubharmonic $G \times K$-invariant function 
$\varrho : \Omega \rightarrow \mathbb R$
such that
$ \pi^{\ast}_K \eta = (\iota_{G \times S})^{\ast} (d^c \varrho).$
The moment map
$ \mu : \Omega \rightarrow \mathrm{Lie}(K)^{\ast},
            x \mapsto ( \xi \mapsto d^c \varrho (\xi_{\Omega} (x))),$
is defined for the $K$-action on $\Omega$. 
%defines the K\"ahlerian reduction by the homeomorphism
%$$ \modulo{\mu^{-1}_K (0)}{K} \cong \ktquot{\Omega}{K}$$
For the existence of the following quotients, %and the homeomorphism property, 
note that the relevant groups $K$ and $K^{\mathbb C}$ respectively act freely.
Thanks to Proposition \ref{Kempf-Ness-free-case}
the diagram
\[ \begin{array}{rcccl} 
     \mu^{-1}(0) & %\stackrel{\iota_{\mu^{-1}(0) \hookrightarrow \Omega}}{\longhookrightarrow} 
                   \stackrel{\iota_{\mu}}{\longhookrightarrow}  & 
     \Omega \subset G^{\mathbb C} \times S^{\mathbb C} \\
     & & \\
     \mbox{ }^{ \pi_{\mu^{-1}(0)} } \Big \downarrow & & 
      \Big \downarrow \mbox{ }^{\pi_{K^{\mathbb C}} } \\
     & & \\
     \modulo{\mu^{-1}(0)}{K} & \stackrel{\kappa}{\longrightarrow} & 
     G^{\mathbb C} \times^{K^{\mathbb C}} S^{\mathbb C}
   \end{array} \]
commutes. 
It shows that there is a canonically defined complex structure on 
$ M^c := \modulo{\mu^{-1}(0)}{K}$.  % $ \cong \ktquot{\Omega}{K},$
Note that the $G$-action on $\Omega$
induces a natural $G$-action by holomorphic transformations on the quotient
$M^c = \modulo{\mu^{-1}(0)}{K}$, 
because the $G$-action and the $K$-action on $\Omega$ commute. 
The function
%$$ \varrho_{\mathrm{red}} :  
%       \modulo{\mu^{-1}(0)}{K} \cong \ktquot{G^c \times S^c}{K} 
%   \rightarrow \mathbb R$$
$ \varrho_{\mathrm{red}} :  
   M^c = \modulo{\mu^{-1}(0)}{K} % \cong \ktquot{\Omega}{K} 
          \rightarrow \mathbb R $
which is induced by 
%$ (\iota_{\mu^{-1}(0) \hookrightarrow \Omega})^{\ast} \varrho$
$ (\iota_{\mu})^{\ast} \varrho$
is $G$-invariant
%$$\pi_{\modulo{\mu_K^{-1}(0)}{K}} : 
%               \mu_K^{-1}(0) \rightarrow \modulo{\mu_K^{-1}(0)}{K} $$ 
and has the property
%$$ (\iota_{\mu^{-1}(0) \hookrightarrow \Omega})^{\ast} \varrho 
%      = (\pi_{\mu^{-1}(0)})^{\ast} (\varrho_{\mathrm{red}}).$$ 
$$ (\iota_{\mu})^{\ast} \varrho 
      = (\pi_{\mu^{-1}(0)})^{\ast} (\varrho_{\mathrm{red}}).$$ 
Then the $G \times K$-equivariant embedding 
$ %\iota_{G \times S \hookrightarrow \Omega} 
    \iota_{G \times S} : G \times S \hookrightarrow \Omega 
                        \subset G^{\mathbb C} \times S^{\mathbb C}$
induces a $G$-equivariant embedding
$ \iota_M : M = G \times^K S \hookrightarrow 
             % \ktquot{\Omega}{K} = \ktquot{G^c \times S^c}{K}
            G^{\mathbb C} \times^{K^{\mathbb C}} S^{\mathbb C}.$
The strictly plurisubharmonic function 
$\varrho_{\mathrm{red}}$ on
$\modulo{\mu^{-1} (0)}{K}$  % $ \cong \ktquot{\Omega}{K} $
%(again called $\varrho$ here) 
has the property that
$$ (\iota_M)^{\ast} (d^c \varrho_{\mathrm{red}}) = \eta.$$ 
To see this, consider the following commutative diagram:
\[ \begin{array}{rcccccl}
      G \times S \phantom{x} 
      & %\stackrel{\iota_{G \times S \hookrightarrow \mu^{-1}(0)}}{\longhookrightarrow} 
        \stackrel{\iota_{G \times S}}{\longhookrightarrow} 
      & \mu^{-1}(0) 
      & %\stackrel{\iota_{\mu^{-1}(0) \hookrightarrow \Omega}}{\longhookrightarrow} 
        \stackrel{\iota_{\mu}}{\longhookrightarrow} 
      & \Omega \subset G^{\mathbb C} \times S^{\mathbb C} \\
      & & & & \\
      \Big \downarrow \mbox{ }^{ \pi_K } & & \phantom{xxx} \Big \downarrow \mbox{ }^{ \pi_{\mu^{-1}(0)} }
                            & & \Big \downarrow \mbox{ }^{ \pi_{K^{\mathbb C}} } \\
      & & & & \\
      M = G \times^K S & \stackrel{\iota_M}{\longhookrightarrow} & M^c = \modulo{\mu^{-1}(0)}{K} 
                 & \stackrel{\kappa}{\longrightarrow} & 
      M^{\mathbb C} = G^{\mathbb C} \times^{K^{\mathbb C}} S^{\mathbb C} 
   \end{array}  \]
Since
%\begin{align*}
%\pi_K^{\ast} \eta
% & = (\iota_{G \times S \hookrightarrow \Omega})^{\ast} (d^c \varrho) \\
% & = (\iota_{G \times S \hookrightarrow \mu^{-1}(0)})^{\ast} 
%     ((\iota_{\mu^{-1}(0) \hookrightarrow \Omega})^{\ast} (d^c \varrho)) \\
% & = (\iota_{G \times S \hookrightarrow \mu^{-1}(0)})^{\ast} 
%     ((\pi_{(\mu)^{-1}(0)})^{\ast} (d^c \varrho_{\mathrm{red}})) \\
% & =\pi_K^{\ast} (\iota_M^{\ast}(d^c \varrho_{\mathrm{red}}))
%\end{align*}
$ (\pi_K)^{\ast} \eta = (\iota_{G \times S})^{\ast} (d^c \varrho)
                    = (\pi_K)^{\ast} ((\iota_M)^{\ast}(d^c \varrho_{\mathrm{red}})) $
surjectivity of $\pi_K$ implies that 
$\eta = (\iota_M)^{\ast}(d^c \varrho_{\mathrm{red}})$.
\end{proof}

\bigskip

\subsection{Complexifications of contact and symplectic manifolds}\label{Complexifications of contact and symplectic manifolds}

\bigskip

In the case where $M$ is a contact manifold
Theorem \ref{Komplexifizierungssatz}
can be reformulated in the sense 
that the $1$-form $\eta$ can be extended 
to a $1$-form $\eta^c$, e.g. $\eta^c := d^c \varrho$, 
on a Stein $G$-complexification $M^c$:

\bigskip

{\it Every contact manifold $(M, \eta)$
with a proper $G$-action 
of a Lie group $G$ with finitely many connected components
can be complexified
equivariantly to a Stein $G$-complexi\-fication $M^c$
with a $G$-invariant $1$-form $\eta^c$ such that
$\iota_M^{\ast}(\eta^c) = \eta$
for the embedding $\iota_M : M \hookrightarrow M^c$.}

\bigskip

%An analogous 
A similar
result for symplectic manifolds is proved by 
Stratmann (\cite{Stratmann}). %(\cite{Stratmann}, Theorem 4.12, p. 26).

%{\it Every symplectic manifold $(M, \omega)$
%with a proper action 
%of a Lie group $G$ with finitely many components
%and an  invariant form $\omega$
%can be $G$-equivariantly complexified to 
%a Stein $G$-manifold $M^c$ such that
%$$ \iota_M^{\ast}(\omega^c) = \omega$$
%for a suitable $G$-invariant K\"ahlerian form $\omega^c$ on $M^c$.}

\bigskip

A contact manifold $(M, \eta)$ can be symplectified, 
i.e., it can be extended 
naturally to a symplectic manifold:
If $(M, \eta)$ is a $(2n+1)$-dimensional contact manifold, 
the two-form 
$$ d(e^t \eta + dt) = e^t dt \wedge \eta + e^t d\eta $$
on $M \times \mathbb R$ is symplectic. 
%because
%$$ (e^t dt \wedge \eta + e^t d\eta)^{n+1} 
%  = e^{t \cdot (n+1)} dt \wedge \eta \wedge (d\eta)^n \not= 0. $$
Here, $t$ denotes the standard coordinate 
on the $\mathbb R$-factor of $M \times \mathbb R$.
A contact-form $\eta$ on $M$ induces a symplectic form 
$\omega = d(e^t \cdot (\pi_{M \times \mathbb R \rightarrow M})^{\ast}(\eta))$,
where $t$ is the coordinate on $\mathbb R$ and
$\pi_{M \times \mathbb R \rightarrow M}$
%$\pi_{M \times \mathbb R \rightarrow M}(m, t) = m$
projects on the first factor.
The complex extension of $M$ to $M^c$
induces a complex extension of $M \times \mathbb R$ to $M^c \times \mathbb C$.
%
%
%This means that 
%$(\iota_{M^{c}})^{\ast} (d^c \varrho)$
%extends $\eta$ on $M^{c}$
%and $dd^c \varrho$ %is a complexification of the symplectization
%%$(M \times \mathbb R, d(e^t \eta + dt) )$ on $\Omega$.
%extends $d(e^t \eta + dt)$ on $\Omega$.
%
%
This means that $\eta$ extends to
$(\iota_{M^{c}})^{\ast} (d^c \varrho)$ on $M^{c}$
and $d(e^t \eta + dt)$ extends to $dd^c \varrho$ on $\Omega$.

%Every contact manifold $(M, \eta)$
%has the property that
%its canonical symplectification 
%$(M \times \mathbb R, d(e^t \eta + dt) )$
%can be complexified to a Stein space $\Omega$
%such that for a certain strictly plurisubharmonic function
%$ \varrho : \Omega \rightarrow \mathbb R $
%one has the property that
%$ \iota_{M \times \mathbb R}^{\ast} (dd^c \varrho) 
%   = d(e^t \eta + dt) $.

\bigskip

%The following Proposition explains 
%why complex extensions and symplectifications are compatible.
The symplectification is compatible with the extension to 
complexifications in the following sense.

\bigskip

\begin{prop}\label{symp-d^c-rho}
Let $(M,\eta)$ be a smooth contact manifold. 
Then there is a Stein complexification $M^{c}$ of $M$
and an open neighbourhood $\Omega$ of $M \times \mathbb{R}$
in $M^{c} \times \mathbb{C}$ such that 
there exists a strictly plurisubharmonic function
$\varrho : \Omega \rightarrow \mathbb R$ for which 
$$ (\iota_{M \times \mathbb R \hookrightarrow M^c \times \mathbb C})^{\ast} (dd^c \varrho)
     = d( e^t \eta + dt ) $$ 
for the embeddings
$ \iota_{M^{c}} : M^{c} \hookrightarrow M^{c} \times \mathbb C, z \mapsto (z, 0),$
and $\iota_M : M \hookrightarrow M^{c}$
$$ \iota_{M}^{\ast}( (\iota_{M^{c}})^{\ast} (d^c \varrho) ) = \eta.$$
%\end{theorem}
\end{prop}

%\bigskip

\begin{proof}[{\bf Proof}]
There is a complexification $M^{c}$ of $M$ 
and a strictly plurisubharmonic function
$\varrho_M : M^{c} \rightarrow \mathbb R$
such that 
$\iota^{\ast}_M (d^c \varrho_M) = \eta $.
Then the function
\[ \begin{array}{rcl}   
   \varrho : M^{c} \times \mathbb C & \rightarrow & \mathbb R \\ 
             (m, z) & \mapsto & e^{\mathrm{Re}(z)} \cdot \varrho_M (m) 
        \end{array} \] 
has the property
$(d^c \varrho) = e^t \cdot d^c \varrho_M - \varrho_M \cdot e^t ds $
where $z=t+is$.
In particular, 
$ \iota_M^{\ast} ((\iota_{M^{c}})^{\ast}(d^c \varrho)) 
    = \iota_M^{\ast} (d^c \varrho_M) 
    = \eta $
and 
$$ (\iota_{M \times \mathbb R})^{\ast} (d^c \varrho) 
    = (\iota_{M \times \mathbb R})^{\ast} (e^t \cdot d^c \varrho_M - \varrho_M \cdot e^t ds)
    = e^t (\iota_{M \times \mathbb R})^{\ast} (d^c \varrho_M) 
    = e^t \eta .$$
If $\nu : M^{c} \rightarrow \mathbb R$
is a strictly plurisubharmonic %exhaustion 
function 
with the property 
$\iota^{\ast}_M (d^c \nu) = 0$
and 
$\iota^{\ast}_M (d \nu) = 0$, 
%Lemma \ref{Kahler-Lemma}    
%(Lemma 2 in \cite{HHL})
%Proposition \ref{compl-dc-rho-eta}
Proposition \ref{compact-extension}
can be applied to 
\[ \begin{array}{rcl}   
   \tilde \nu : M^{c} \times \mathbb C & \rightarrow & \mathbb R \\ 
                (m, z) & \mapsto & \nu (z) + |z|^2 
        \end{array} \] 
and to $\varrho$ 
to obtain a strictly plurisubharmonic function
$\varrho: \Omega \rightarrow \mathbb R$
on a Stein neighbourhood $\Omega$ of 
$M \times \mathbb R$ in $M^{c} \times \mathbb C$.
\end{proof}

\bigskip

Proposition \ref{symp-d^c-rho} also has an equivariant version:

\bigskip

\begin{cor}\label{cor-symp-d^c-rho}
If $G \times M \rightarrow M$ is a proper $G$-action, 
there is a proper extension to $M^c$: 
The trivial extension to an action on
$ M^c  \times \mathbb C$ defines equivariant embeddings
$$ \iota_{M^c} : M^c \hookrightarrow M^c \times \mathbb C \mbox{ and }
   \iota_M : M \hookrightarrow M^c$$
such that $\varrho : \Omega \rightarrow \mathbb R$
can be chosen to be strictly plurisubharmonic and $G$-invariant on
$\Omega \subset M^c \times \mathbb C$.
\end{cor}

%\bigskip

\begin{proof}[{\bf Proof}]
It has just to be observed that 
in the proof of Proposition \ref{symp-d^c-rho} 
the function
$\varrho_M : M^c \rightarrow \mathbb R$
can be chosen to be $G$-invariant 
by Theorem \ref{Komplexifizierungssatz}
and as a strictly plurisubharmonic function.
\end{proof}

\bigskip

\begin{cor}\label{symp-Kahler-form}
Let $M$ be a real analytic manifold with a contact form $\eta$.
Then there is a
Stein complexification $M^{c}$ of $M$
and an open neighbourhood $\Omega$ of $M \times \mathbb{R}$
in $M^{c} \times \mathbb{C}$ such that
the symplectic form $\omega:= d(e^{t} \eta + dt)$
is the pull-back $(\iota_{M \times \mathbb{R}})^{\ast}(\beta)$
of a K\"ahler form $\beta$ on $\Omega$.
%\end{theorem}
\end{cor}

%\bigskip

\begin{proof}[{\bf Proof}]
This is a consequence of Proposition \ref{symp-d^c-rho}
because for a strictly plurisubharmonic function 
$\varrho : \Omega \subset M^{c} \times \mathbb C 
                                \rightarrow \mathbb R$, 
$\beta := dd^c \varrho$ is a K\"ahler form 
with the properties %which are 
stated in Corollary \ref{symp-Kahler-form}.
\end{proof}

\bigskip

\begin{rem}
Similarly to the equivariant statemant in Corollary \ref{cor-symp-d^c-rho},
an equivariant version of 
Corollary \ref{symp-Kahler-form} can be formulated.
\end{rem}

\bigskip

\section{Compatibility of reductions}\label{Compatibility of reductions}

\bigskip

In this section, the compatibility of the complexification
with reductions by symmetries is discussed.
Roughly speaking, the guiding question is 
whether the K\"ahlerian reduction of a complexification
of a contact manifold can be regarded as the complexification 
of the contact reduction.

\medskip

Throughout this section 
$(M, \eta)$ is assumed to be a contact manifold 
on which an extendable Lie group $G$ 
with finitely many connected components
acts properly by contact transformations, i.e.,
by leaving $\eta$ invariant.
%Let $\eta$ be a $G$-invariant contact form on $M$ and 
Fix a $G$-invariant smooth Stein complexification $M^c$ of $M$
such that $\eta = \iota_M^{\ast}(d^c \varrho)$ holds
for some smooth $G$-invariant strictly plurisubharmonic function
$\varrho : M^c \rightarrow \mathbb R$ 
(see Theorem \ref{Komplexifizierungssatz}). 
Furthermore, 
assume that there is a globalization $M^{\mathbb C}$
of the local $G^{\mathbb C}$-action on $M^c$
such that $M^c$ is openly and $G$-equivariantly embedded 
in the $G^{\mathbb C}$-manifold $M^{\mathbb C}$.

\bigskip

\subsection{Compatibility of moment maps for free actions}\label{Compatibility of moment maps for free actions}

\bigskip

%Let $G$ be an extendable Lie group with finitely many connected components
%which acts in a proper %and free 
%fashion on a contact manifold $(M, \eta)$
%and leaves $\eta$ invariant.
%Let $M^c$ be a $G$-equivariant complexification of $M$
%and $\varrho :M^c \rightarrow \mathbb R$
%be a strictly plurisubharmonic and $G$-invariant function 
%such that the embedding 
%$ \iota_M : M \hookrightarrow M^c $
%satisfies $ \iota_M^{\ast}(d^c \varrho) = \eta$ (see Theorem \ref{Komplexifizierungssatz}).
%%\footnote{This is guaranted by Theorem \ref{Komplexifizierungssatz}.}
%%Proposition \ref{Komplexifizierung-1-Formen}.
%%By Corollary \ref{mu-M-in-mu-Mc},
%%The momentum zero set of $M$
%%embeds in $(\mu_{M^c})^{-1}(0)$ by
%%$$ \iota_M : \mu_M^{-1}(0) \hookrightarrow (\mu_{M^c})^{-1}(0).$$

Under the assumptions stated at the beginning of the section,
there exists a moment map %$\mu_M$ 
on the contact manifold %$M$
\[ \begin{array}{rcl}
    \mu_M : M & \rightarrow & \mathfrak{g}^{\ast} \\
            m & \mapsto & \big ( \xi \mapsto  
                 \eta (\xi_M (m))
                 = \eta (\frac{d}{dt} \orbit{\exp(t \xi)}{m} \vert_{t=0}) \big )
   \end{array} \]
and a moment map %$\mu_{M^c}$ 
on the K\"ahler manifold %$M^c$
\[ \begin{array}{rcl}
    \mu_{M^c} : M^c & \rightarrow & \mathfrak{g}^{\ast} \\
            x & \mapsto & ( \xi \mapsto  
                          d^c \varrho(\xi_{M^c}(x)) ).
   \end{array} \]
The relation $\eta = \iota_M^{\ast}(d^c \varrho)$
implies that the K\"ahlerian moment map extends the contact moment map, i.e., 
$\mu_{M^c} \circ \iota_M = \mu_M$.
%or in other words $\mu_M = \mu_{M^c \vert_M}$.
Cauchy-Riemann geometry enters the picture, 
because the hypersurface 
$M^{\mathrm{\mathrm{CR}}} = \varrho^{-1}(0)$
plays a role as it contains $M$.
This fact makes use of the assumption that 
the K\"ahlerian moment map is defined by the potential $\varrho$.

\bigskip

\begin{lem}\label{lem-Stein-tube}
Let $(M^c, d^c \varrho)$ be a 
%Stein 
complexification of a contact manifold $(M, \eta)$ and 
$\varrho : M^c \rightarrow \mathbb{R}$ 
be a strictly plurisubharmonic function 
with $M \subset \varrho^{-1}(0)$ 
such that $d^c \varrho$ extends $\eta$ 
in the sense that 
$\eta = \iota_M^{\ast}(d^c \varrho)$.
%There is a Stein domain $\Omega \subset M^c$, 
%$M \subset \Omega$, such that 
%$\varrho^{-1}(0) \cap \Omega$
%is a smooth hypersurface.
\begin{enumerate}
\item \label{lem-Stein-tube-a}
Then possibly after shrinking $M^c$ to 
a smaller %Stein 
neighbourhood of $M$,
$\varrho^{-1}(0)$
is a smooth hypersurface in $M^c$.
\item \label{levi-convex}
%\begin{cor}\label{levi-convex}
%In the situation of Lemma \ref{lem-Stein-tube}, %on $\Omega$,
%$M^{\mathrm{CR}} := \varrho^{-1}(0)$
%is a Levi convex hypersurface.%\footnote{see \ref{The contact and Cauchy-Riemann symmetry of strictly pseudoconvex hypersurfaces}.}.
%\end{cor}
The smooth hypersurface $M^{\mathrm{CR}}:= \varrho^{-1}(0)$
%is Levi convex.
is a strongly pseudoconvex hypersurface.
\end{enumerate}
\end{lem}

%\medskip

\begin{proof}[{\bf Proof}] %of Lemma \ref{lem-Stein-tube}}]
Since $\eta$ is nowhere vanishing on $M$
and it is the pull-back of $d^c \varrho$, 
a) %\ref{lem-Stein-tube-a}
follows, because
it is immediate that $d \varrho$ vanishes nowhere
in a neighbourhood of $M$. 
The statement b) %\ref{levi-convex}
 is just a matter of definitions.
\end{proof}

\bigskip

The action of $G$ leaves $M^{\mathrm{CR}} = \varrho^{-1}(0)$
invariant and the inclusions
$$ (M, \eta) \phantom{x} \hookrightarrow \phantom{x}
   \big( M^{\mathrm{CR}}, d^c \varrho \vert_{M^{\mathrm{CR}}} \big) 
    \phantom{x} \hookrightarrow \phantom{x} (M^c, d^c \varrho) $$
are all $G$-equivariant.
Assume that the Lie subgroup $L$ of $G$ acts freely (and properly) on 
the contact manifold $(M, \eta)$ 
and leaves $\eta$ invariant.
%
%
%The following Proposition introduces 
%a Cauchy-Riemann moment map $\mu_{M^{\mathrm{CR}}}$ 
%for the hypersurface $M^{\mathrm{CR}}$,
%which was introduced by Loose (\cite{Loose-CR-reduction}),
%and provides a proof that $\mu_{M^{\mathrm{CR}}}$
%can be identified with $\mu_{M^c \vert_{M^{\mathrm{CR}}}}$.
%
%
%The following Proposition shows the compatibility 
%of the three moment maps.
%
%
In the following proposition it is shown that 
in the setting of this work, the restriction 
$\mu_{M^c \vert_{M^{\mathrm{CR}}}}$
can be regarded as the Cauchy-Riemann moment map
defined in \cite{Loose-CR-reduction}.
This involves the natural projection 
$ \alpha_p : T_p M^{\mathrm{CR}} \rightarrow \modulo{T_p M^{\mathrm{CR}}}{H_p}, $
where $H_p = T_p M \cap J(T_p M)$.
%
%
%Since $M^{\mathrm{CR}}$
%is a hypersurface given as the zero set of a 
%strictly plurisubharmonic function, 
%every hyperplane 
%$H_p = \lbrace v \in T_p (M^{\mathrm{CR}}) \vert d^c \varrho(v) = 0 \rbrace$.
%Note that
%$ T_p M^{\mathrm{CR}} = \lbrace v \in T_p M^c \vert d \varrho(v) = 0 \rbrace $.
%
%
It follows from the definition of the operator $d^c$
that the Cauchy-Riemann tangent space can be described by
$H_p = \lbrace v \in T_p (M^{\mathrm{CR}}) \vert d^c \varrho(v) = 0 \rbrace$.
Let $H = \cup_{p \in M^{\mathrm{CR}}} H_p$ be the 
%hyperplane bundle 
Cauchy-Riemann bundle of hyperplanes
and $B$ denote the (real) line bundle $B = \modulo{T M^{\mathrm{CR}}}{H}$.
Then $\alpha$ can be considered as a $B$-valued $1$-form
which defines the Cauchy-Riemann moment map
\[ \begin{array}{rcl}
  \mu_{M^{\mathrm{CR}}} : M^{\mathrm{CR}} & \rightarrow & \mathrm{Lie}(L)^{\ast} \otimes B \\
            p & \mapsto & \big( \xi \mapsto \alpha_p(\xi_{M^{\mathrm{CR}}}(p)) \big)
\end{array} \]
for every $\xi \in \mathrm{Lie}(L)$.

\bigskip

\begin{prop}\label{prop-CR-moment}
Let $L$ act freely and properly on $M$ and $M^c$.
For the inclusions
$$ (M, \eta) \phantom{x} \hookrightarrow \phantom{x}
   \big( M^{\mathrm{CR}}, d^c \varrho \vert_{M^{\mathrm{CR}}} \big) 
   \phantom{x} \hookrightarrow \phantom{x} (M^c, \eta^c = d^c \varrho) $$
the K\"ahlerian moment map 
\[ \begin{array}{rcl}   
    \mu_{M^c} : M^c & \rightarrow & \mathrm{Lie}(L)^{\ast} \\ 
           p & \mapsto & \big( \xi \mapsto (d^c \varrho) (\xi_{M^c}(p))  
                          = \eta^c (\xi_{M^c}(p)) \big)
        \end{array} \] 
has the property that its restriction 
$ \mu_{M^{\mathrm{CR}}} := \mu_{M^c \vert_{M^{\mathrm{CR}}}} $
%can be identified with 
is the Cauchy-Riemann moment map for the 
$L$-action on $M^{\mathrm{CR}}$, 
if $\modulo{TM^{\mathrm{CR}}}{H}$ is trivialized by the mapping
$\modulo{TM^{\mathrm{CR}}}{H} \rightarrow \mathbb R$, $\alpha_p(v) \mapsto (d^c \varrho)(v)$. 
%and 
%$ \mu_M := \mu_{M^c} \vert_{M}$
%is the contact moment map on $M$ induced by $\eta$ 
%\[ \begin{array}{rcl}   
%    \mu_M : M & \rightarrow & \mathfrak{l}^{\ast} \\ 
%            m & \mapsto & ( \xi \mapsto \eta (\hat \xi(m))).
%        \end{array} \] 
\end{prop}

%\smallskip

\begin{proof}[{\bf Proof}] %of Proposition \ref{prop-CR-moment}}]
The line bundle $B = \modulo{T M^{\mathrm{CR}}}{H}$
is trivializable in this situation by the map
\[ \begin{array}{rcl}   
    \modulo{T M^{\mathrm{CR}}}{H} & \rightarrow & \mathbb R \\ 
     \alpha_p(v) & \mapsto & d^c \varrho (v). 
        \end{array} \] 
It is well-defined because if 
$\alpha_p(v) = \alpha_p(w)$, $v-w \in H$ and 
$(d^c \varrho)(v-w) = 0 $ and therefore
$ (d^c \varrho)(v) %  = (d^c \varrho)(v-w + w) 
     = (d^c \varrho)(v-w) + (d^c \varrho)(w) 
     = (d^c \varrho)(w). $
Under this trivialization, 
$\mathrm{Lie}(L)^{\ast} \otimes B \cong \mathrm{Lie}(L)^{\ast} $
with the identification
\[ \begin{array}{rcl}   
  \mathrm{Lie}(L)^{\ast} \otimes B & \rightarrow & \mathrm{Lie}(L)^{\ast} \\
    \big( \xi \mapsto \alpha_p(\xi_{M^{\mathrm{CR}}}) \big) & \mapsto & 
    \big( \xi \mapsto d^c \varrho (\xi_{M^{c}}(p)) \big)
 \end{array} \] 
defines the Cauchy-Riemann moment map 
$\mu_{M^{\mathrm{CR}}} : M^{\mathrm{CR}} \rightarrow \mathrm{Lie}(L)^{\ast}$
by 
$$\mu_{M^{\mathrm{CR}}, \xi}(p) 
    = (\iota_{M^{\mathrm{CR}}})^{\ast}(d^c \varrho)(\xi_{M^{\mathrm{CR}}}(p))
    = d^c \varrho (\xi_{M^{c}} (p))$$
for $\xi \in \mathrm{Lie}(L)$,
where $\iota_{M^{\mathrm{CR}}} : M^{\mathrm{CR}} \hookrightarrow M^c$ 
embeds $M^{\mathrm{CR}}$ into $M^c$.
\end{proof}

\bigskip

\subsection{Cauchy-Riemann, contact and K\"ahlerian reductions}\label{CR-contact-Kahler-red}

\bigskip

It will be shown later that 
the reduction along suitable strata of orbit types 
can be described by quotients of free actions on certain submanifolds.
This is why in this subsection, 
the case of a freely acting Lie group $L$ is considered.
The properties $\mu_M = \mu_{M^{\mathrm{CR}} \vert_{M}}$
and $\mu_{M^{\mathrm{CR}}} = \mu_{M^{c} \vert_{M^{\mathrm{CR}}}}$
yield the inclusions
of the momentum zero levels 
%$ (\mu_M)^{-1}(0)$, $ (\mu_{M^{\mathrm{CR}}})^{-1}(0)$ and $ (\mu_{M^c})^{-1}(0)$
$$ (\mu_{M})^{-1} (0) \hookrightarrow 
  (\mu_{M^{\mathrm{CR}}})^{-1} (0) \hookrightarrow 
  (\mu_{M^c})^{-1} (0).$$
%$$\momentM \hookrightarrow \momentMCR \hookrightarrow \momentMc. $$
The inclusion of $(\mu_M)^{-1}(0)$ and of $(\mu_{M^{\mathrm{CR}}})^{-1}(0)$
in %the momentum zero level 
$(\mu_{M^c})^{-1}(0)$
will be examined more closely in the following.

\bigskip

\medskip

{\bf Contact and K\"ahlerian reduction}

\bigskip

Now the connection between the contact reduction of $M$ 
and the K\"ahlerian reduction of $M^c$
with respect to a freely and properly acting group $L$
is studied.
%
%
%
%The set-up can be summerized in the following diagram:
The situation for the embedding of $M$ in $M^c$ 
can be summarized in the following diagram:
\[ \begin{array}{rrcccccll}
     M \phantom{xxxxxx} & \stackrel{\iota_M}{\longhookrightarrow} \phantom{xx} &  M^c \phantom{xxx} \\
       & & \\
  \phantom{xxx}   \longhookuparrow \mbox{ }^{ \iota_{(\mu_M)^{-1}(0)} } & &  
  \phantom{xxx}   \longhookuparrow \mbox{ }^{ \iota_{(\mu_{M^c})^{-1}(0)} } \\ 
       & & \\
      (\mu_M)^{-1}(0) \phantom{xx}
      & \stackrel{\iota_{M} \vert_{(\mu_M)^{-1}(0)}}{\longhookrightarrow} & 
     (\mu_{M^c})^{-1}(0) \\
      & & \\
    \Big \downarrow \mbox{ }_{ \pi_{(\mu_M)^{-1}(0)} } & & 
   \phantom{xxx} \Big \downarrow \mbox{ }_{ \pi_{(\mu_{M^c})^{-1}(0)} } \\
      & & \\
     \modulo{(\mu_M)^{-1}(0)}{L} & \stackrel{\iota_{(\mu_M)^{-1}(0) / L}}{\longrightarrow} & 
     \modulo{(\mu_{M^c})^{-1}(0)}{L} .
   \end{array} \]
Note that Proposition \ref{free-proper-case}
shows that for this case described here, 
$\modulo{(\mu_{M^c})^{-1}(0)}{L}$
is a complex manifold and the function $\varrho_{\mathrm{red}}$,
defined by 
$\varrho_{\mathrm{red}} \circ \pi_{(\mu_{M^c})^{-1}(0)} 
   = \varrho \circ \iota_{(\mu_{M^c})^{-1}(0)} $,
is a K\"ahlerian potential. 
It can be checked that the mapping 
%\[ \begin{array}{rcl}
%       \iota_{(\mu_M)^{-1}(0) / L} : \modulo{(\mu_M)^{-1}(0)}{L} 
%               & \rightarrow & \modulo{(\mu_{M^c})^{-1}(0)}{L} \\
%       \pi_{(\mu_M)^{-1}(0)}(m) & \mapsto & \pi_{(\mu_{M^c})^{-1}(0)} (\iota_M \vert_{(\mu_M)^{-1}(0)} (m)) 
%   \end{array} \]
$ \iota_{(\mu_M)^{-1}(0) / L}$ 
%which sends $\pi_{(\mu_M)^{-1}(0)}(m)$ to 
%$ \pi_{(\mu_{M^c})^{-1}(0)} (\iota_M \vert_{(\mu_M)^{-1}(0)} (m)) $,
is well-defined.
In Proposition \ref{prop-contact-reduction-varrho},
it is shown that
$(\iota_{(\mu_M)^{-1}(0) / L})^{\ast} (d^c \varrho_{\mathrm{red}})$
is the unique
$1$-form $\eta_{\mathrm{red}}$ on $\modulo{(\mu_M)^{-1}(0)}{L}$
with the property 
\begin{equation}\label{contact-reduction-property-a}
(\pi_{(\mu_M)^{-1}(0)})^{\ast}(\eta_{\mathrm{red}}) = (\iota_{(\mu_M)^{-1}(0)})^{\ast}(\eta).
\end{equation}
%holds.
The manifold $\modulo{(\mu_M)^{-1}(0)}{L}$
with the unique $1$-form $\eta_{\mathrm{red}}$
such that 
%$(\pi_{(\mu_M)^{-1}(0)})^{\ast}(\eta_{\mathrm{red}}) = (\iota_{(\mu_M)^{-1}(0)})^{\ast}(\eta)$
(\ref{contact-reduction-property-a})
holds is called %in \cite{Loose-contact-reduction} and \cite{Willett}
the contact reduction of $(M, \eta)$ 
as defined in \cite{Loose-contact-reduction} and \cite{Willett}.

\bigskip

\begin{prop}\label{prop-contact-reduction-varrho}
Let the extendable Lie group $L$ with finitely many connected components 
act freely and properly on $M$. 
Then 
%the  $L$-complexi\-fication $(M^c, \varrho)$
%provides the contact reduction %described in \cite{Loose-contact-reduction} 
%with the $1$-form 
$ (\iota_{(\mu_M)^{-1}(0) / L})^{\ast} (d^c \varrho_{\mathrm{red}})$ 
is the unique $1$-form $\eta_{\mathrm{red}}$
on $\modulo{(\mu_M)^{-1}(0)}{L}$ 
such that 
$$(\pi_{(\mu_M)^{-1}(0)})^{\ast}(\eta_{\mathrm{red}}) = (\iota_{(\mu_M)^{-1}(0)})^{\ast}(\eta).$$
\end{prop}

%\bigskip

\begin{proof}[{\bf Proof}]
The assumptions on the action of $L$ imply that 
the geometric quotients
$\modulo{(\mu_M)^{-1}(0)}{G}$ and $\modulo{(\mu_{M^c})^{-1}(0)}{G}$
are manifolds.
The %reduced 
function 
$\varrho_{\mathrm{red}} : \modulo{(\mu_{M^c})^{-1}(0)}{L} \rightarrow \mathbb R $ 
is definied by
%$$ (\pi_{(\mu_{M^c})^{-1}(0)})^{\ast} (\varrho_{\mathrm{red}}) 
%       = (\iota_{(\mu_{M^c})^{-1}(0)})^{\ast} (\varrho) ,$$
%therefore 
$\varrho_{\mathrm{red}} \circ \pi_{(\mu_{M^c})^{-1}(0)} 
  = \varrho \circ \iota_{(\mu_{M^c})^{-1}(0)}$. 
%It is to prove that for the $1$-form 
%$(\iota_{(\mu_M)^{-1}(0) / L})^{\ast} (d^c \varrho_{\mathrm{red}})$ 
The desired result follows from
the identity
\begin{equation}\label{contact-reduction-property}
 \big( \pi_{(\mu_M)^{-1}(0)} \big)^{\ast} \big( (\iota_{(\mu_M)^{-1}(0) / L})^{\ast} 
                   (d^c \varrho_{\mathrm{red}}) \big) 
       = \big( \iota_{(\mu_M)^{-1}(0)} \big)^{\ast} (\eta), 
\end{equation}
%holds. Then 
because
the uniqueness of the contact reduction implies that
the reduced contact structure is defined by the $1$-form
$ (\iota_{(\mu_M)^{-1}(0) / L})^{\ast} (d^c \varrho_{\mathrm{red}} )$.
%defines indeed the reduced contact structure.
%Because of the fact that the diagram commutes,  
%Since
%To prove (\ref{contact-reduction-property}) 
%observe that since
%$ (\iota_{(\mu_M)^{-1}(0) / L}) \circ \pi_{(\mu_M)^{-1}(0)} 
%   = \pi_{(\mu_{M^c})^{-1}(0)} \circ 
%     ( \iota_M \vert_{(\mu_M)^{-1}(0)} )$
%it follows that
%\begin{align*}
%  \big( \pi_{(\mu_M)^{-1}(0)} \big)^{\ast} \big( (\iota_{(\mu_M)^{-1}(0) / L})^{\ast} 
%                  (d^c \varrho_{\mathrm{red}}) \big) 
%%  & = (\iota_{(\mu_M)^{-1}(0) / L} \circ \pi_{(\mu_M)^{-1}(0)} )^{\ast} 
%%      (d^c \varrho_{\mathrm{red}}) \\
%%%  & = (\pi_{(\mu_{M^c})^{-1}(0) / G} \circ 
%%%       \iota_M \vert_{(\mu_M)^{-1}(0)})^{\ast} 
%%%       d^c \varrho_{\mathrm{red}} \\
%%%  & = (\iota_M \vert_{(\mu_M)^{-1}(0)})^{\ast} 
%%%      (\pi_{(\mu_{M^c})^{-1}(0) / G}^{\ast}(d^c \varrho_{\mathrm{red}}) ) \\ 
%  & = \big( \iota_M \vert_{(\mu_M)^{-1}(0)} \big)^{\ast} 
%      \big( (\iota_{(\mu_{M^c})^{-1}(0)})^{\ast}(d^c \varrho) \big)  \\
%%  & = (\iota_{(\mu_{M^c})^{-1}(0)} \circ \iota_M \vert_{(\mu_M)^{-1}(0)} )^{\ast} 
%%      (d^c \varrho)  \\
%%  & = (\iota_M \circ \iota_{(\mu_M)^{-1}(0)})^{\ast} (d^c \varrho) \\
%  & = \big( \iota_{(\mu_M)^{-1}(0)} \big)^{\ast} \big( (\iota_M)^{\ast} (d^c \varrho) \big) \\
%  & = \big( \iota_{(\mu_M)^{-1}(0)} \big)^{\ast} (\eta),
%\end{align*}
%%and this was to show.
%%\end{proof}
%and hence, $d^c \varrho_{\mathrm{red}}$ 
%fulfills property (\ref{contact-reduction-property}).
Since %properties (\ref{contact-reduction-property-a}) 
property (\ref{contact-reduction-property-a}) 
%and (\ref{contact-reduction-property-b})
holds for the %definite 
unique $1$-form 
$\eta_{\mathrm{red}}$, the $1$-form 
$ (\iota_{(\mu_{M})^{-1}(0) / L})^{\ast} (d^c \varrho_{\mathrm{red}})$
agrees with $\eta_{\mathrm{red}}$ which provides 
the contact reduction 
$\Big( \modulo{(\mu_M)^{-1}(0)}{L}, \eta_{\mathrm{red}} \Big)$.
\end{proof}

\bigskip

\begin{cor}
The $2$-form 
$\omega_{\mathrm{red}} := d \eta_{\mathrm{red}}$
on $\modulo{(\mu_M)^{-1}(0)}{L}$ satisfies
$$ (\pi_{(\mu_M)^{-1}(0)})^{\ast} \omega_{\mathrm{red}} 
  = (\iota_{(\mu_M)^{-1}(0)})^{\ast} d\eta.$$
\end{cor}

%\bigskip

\begin{proof}[{\bf Proof}]
It follows from Proposition \ref{prop-contact-reduction-varrho} 
that
%\begin{align*}
% (\iota_{(\mu_M)^{-1}(0)})^{\ast} d\eta
%   & = d (\iota_{(\mu_M)^{-1}(0)})^{\ast}(\eta)  \\
%%   & = d ( (\pi_{(\mu_M)^{-1}(0)})^{\ast} ((\iota_{(\mu_M)^{-1}(0) / L})^{\ast} 
%%                          d^c \varrho_{\mathrm{red}}) ) \\
%   & = (\pi_{(\mu_M)^{-1}(0)})^{\ast} ( (\iota_{(\mu_M)^{-1}(0) / L})^{\ast} 
%                       (dd^c \varrho_{\mathrm{red}}) ).
%\end{align*}
$$ (\iota_{(\mu_M)^{-1}(0)})^{\ast} d\eta
    = d (\iota_{(\mu_M)^{-1}(0)})^{\ast}(\eta)  
    = (\pi_{(\mu_M)^{-1}(0)})^{\ast} ( (\iota_{(\mu_M)^{-1}(0) / L})^{\ast} 
                       (dd^c \varrho_{\mathrm{red}}) ). $$
\end{proof}

\bigskip

\medskip

{\bf Cauchy-Riemann and K\"ahlerian reduction}

\bigskip

The following result characterizes both the contact reduction 
and the Cauchy-Riemann reduction of 
$M^{\mathrm{CR}}$ 
as the hypersurface $(\varrho_{\mathrm{red}})^{-1}(0)$ 
in the K\"ahlerian reduced space $\modulo{(\mu_{M^c})^{-1}(0)}{L}$.
The following sketch illustrates 
%the situation:
the setting:
\[ \begin{array}{rccccccccl} 
    M^{\mathrm{CR}} \phantom{xxxxxx} & \stackrel{\iota_{M^{\mathrm{CR}}}}{\longhookrightarrow} & M^c  \phantom{xxxxxx} \\
  & & \\
    \longhookuparrow \mbox{ }^{^{ \iota_{(\mu_{M^{\mathrm{CR}}})^{-1}(0)} }} & & 
    \longhookuparrow \mbox{ }^{^{ \iota_{(\mu_{M^c})^{-1}(0)} }} \\
  & & \\
    (\mu_{M^{\mathrm{CR}}})^{-1} (0)   \phantom{xxx}
    & \stackrel{\iota_{M^{\mathrm{CR}}} \vert_{(\mu_{M^{\mathrm{CR}}})^{-1} (0)}}{\longhookrightarrow}  
      %%\stackrel{\iota_{(\mu_{M^{\mathrm{CR}}})^{-1} (0) \hookrightarrow (\mu_{M^c})^{-1}(0)}}{\longhookrightarrow} 
    &  (\mu_{M^c})^{-1}(0)  \phantom{xxx} \\
  & & \\
   \Big \downarrow \mbox{ }_{ \pi_{(\mu_{M^{\mathrm{CR}}})^{-1} (0)} } & & 
   \Big \downarrow \mbox{ }_{ \pi_{(\mu_{M^c})^{-1} (0)} } \\
  & & \\
%  \ktquot{M^c}{G} \cong [{\bf Es ist ja L = G nicht kompact und daher kein Isomorphismus. }]
    \modulo{(\mu_{M^{\mathrm{CR}}})^{-1}(0)}{L} 
    & \stackrel{\iota_{(\mu_{M^{\mathrm{CR}}})^{-1}(0) / L}}{\longhookrightarrow}  &  
    \modulo{(\mu_{M^c})^{-1}(0)}{L}.
\end{array} \]

%\bigskip

\medskip

\begin{prop}\label{contact-isomorphic}
The hypersurface $(\varrho_{\mathrm{red}})^{-1}(0) \subset \modulo{(\mu_{M^c})^{-1}(0)}{L}$
can be regarded in two ways:
\begin{enumerate}
\item The pull-back of the $1$-form 
      $ d^c \varrho_{\mathrm{red}}$ to  
      $(\varrho_{\mathrm{red}})^{-1}(0)$
      gives $(\varrho_{\mathrm{red}})^{-1}(0)$
      the structure of a contact manifold which 
     %% $ (\iota_{(\varrho_{\mathrm{red}})^{-1}(0) \hookrightarrow (\mu_{M^c})^{-1}(0) / L})^{\ast} (d^c \varrho_{\mathrm{red}})$
      is isomorphic to
      the contact reduced space for the $L$-action on 
      $\big( M^{\mathrm{CR}}, (\iota_{M^{\mathrm{CR}}})^{\ast} (d^c \varrho) \big)$.
\item The hypersurface $(\varrho_{\mathrm{red}})^{-1}(0)$ is isomorphic 
      as a Cauchy-Riemann manifold to the Cauchy-Riemann reduction of $M^{\mathrm{CR}}$
      with respect to $L$.
\end{enumerate}
\end{prop}

\bigskip

\begin{rem}
%By Lemma \ref{lem-B} it is immediately clear that
%$( M^{\mathrm{CR}}_{\mathrm{red}}, 
 % \iota^{\ast}_{M^{\mathrm{CR}}_{\mathrm{red}}} (d^c \varrho_{\mathrm{red}}))$ 
Since $\varrho_{\mathrm{red}}$ is strictly plurisubharmonic,
the form
%$\big( (\varrho_{\mathrm{red}})^{-1}(0), 
%  (\iota_{M^{\mathrm{CR}}_{\mathrm{red}}})^{\ast} (d^c \varrho_{\mathrm{red}}) \big)$
$d^c \varrho_{\mathrm{red}}$ pulled back to $ (\varrho_{\mathrm{red}})^{-1}(0)$  
is a contact form.
\end{rem}

%\bigskip

\smallskip

\begin{proof}[{\bf Proof}] \hfill
\begin{enumerate}
\item
%Loose (\cite{Loose-contact-reduction}) proves that 
As mentioned in Proposition \ref{prop-contact-reduction-varrho} 
there is a unique contact structure 
$\eta_{\mathrm{red}}$ on the reduced space $\modulo{(\mu_{M^{\mathrm{CR}}})^{-1}(0)}{L}$ 
such that the identity
$$ \big( \iota_{(\mu_{M^{\mathrm{CR}}})^{-1}(0)} \big)^{\ast} 
   \big( (\iota_{M^{\mathrm{CR}}})^{\ast} (d^c \varrho) \big)
     = \big( \pi_{(\mu_{M^{\mathrm{CR}}})^{-1}(0)} \big)^{\ast} (\eta_{\mathrm{red}}) $$
holds. 
The commutativity of the diagram above shows that
$$ \big( \pi_{(\mu_{M^{\mathrm{CR}}})^{-1}(0)} \big)^{\ast} 
    \big( (\iota_{(\mu_{M^{\mathrm{CR}}})^{-1}(0) / L})^{\ast} (d^c \varrho_{\mathrm{red}}) \big) 
 = \big( \iota_{M^{\mathrm{CR}}} \circ \iota_{(\mu_{M^{\mathrm{CR}}})^{-1}(0)} \big)^{\ast} 
    (d^c \varrho) $$
and therefore
$ \big( \pi_{(\mu_{M^{\mathrm{CR}}})^{-1} (0)} \big)^{\ast} 
   \big( (\iota_{M^{\mathrm{CR}}})^{\ast} (d^c \varrho_{\mathrm{red}}) \big)   
  = \big( \iota_{M^{\mathrm{CR}}} \circ \iota_{(\mu_{M^{\mathrm{CR}}})^{-1}(0)} \big)^{\ast} 
    (d^c \varrho)  $.
Since the reduced form is the unique $1$-form with this property 
it follows that 
$ (\iota_{(\mu_{M^{\mathrm{CR}}})^{-1}(0) / L})^{\ast} (d^c \varrho_{\mathrm{red}})$
gives the contact structure. 
\item
Let $\varrho_{\mathrm{red}}$
be the function on
the K\"ahlerian reduction 
$\modulo{(\mu_{M^c})^{-1}(0)}{L}$ % \cong \modulo{\mu^{-1}(0)}{L}$
which is induced by the restriction 
$\varrho \vert_{(\mu_{M^c})^{-1}(0)}$.
This is a strictly plurisubharmonic function, 
and if $0$ is a regular value of $\varrho$, 
$0$ remains a regular value of $\varrho_{\mathrm{red}}$.
The map 
$\modulo{(\mu_{M^{\mathrm{CR}}})^{-1}(0)}{L} \hookrightarrow \modulo{(\mu_{M^c})^{-1}(0)}{L} $ 
induces a bijection between 
$\modulo{(\mu_{M^{\mathrm{CR}}})^{-1}(0)}{L}$ and 
$(\varrho_{\mathrm{red}})^{-1}(0)$.
Since the group action on $M^c$ 
is by holomorphic transformations 
and leaves the Cauchy-Riemann hypersurface
$M^{\mathrm{CR}}$
invariant, the induced action on $M^{\mathrm{CR}}$
is by Cauchy-Riemann diffeomorphisms.
The strictly plurisubharmonic function 
$\varrho$ defines a Cauchy-Riemann submanifold
$ \varrho^{-1}(0) \cap (\mu_{M^c})^{-1}(0)$
which is mapped to 
$ (\varrho_{\mathrm{red}})^{-1}(0) \subset \modulo{(\mu_{M^c})^{-1}(0)}{L}$
by the Cauchy-Riemann map
%$$ \pi_{(\mu_{M^{\mathrm{CR}}})^{-1}(0) } : (\mu_{M^{\mathrm{CR}}})^{-1}(0) 
%        \rightarrow \modulo{(\mu_{M^{\mathrm{CR}}})^{-1}(0)}{L}$$
$ \pi_{(\mu_{M^{\mathrm{CR}}})^{-1}(0) }$.
%This is because
%$ (\mu_{M^{\mathrm{CR}}})^{-1}(0) = (\mu_{M^c})^{-1}(0) \cap M^{\mathrm{CR}} $
%and the function $\varrho_{\mathrm{red}}$ satisfies
%$(\pi_{(\mu_{M^c})^{-1}(0)})^{\ast} \varrho_{\mathrm{red}} 
%    = \varrho \vert_{(\mu_{M^c})^{-1}(0)} .$
Since Loose (\cite{Loose-CR-reduction}, Theorem 1.2)
proves that the projection 
$ \pi_{(\mu_{M^{\mathrm{CR}}})^{-1}(0) } $ 
%$: (\mu_{M^{\mathrm{CR}}})^{-1}(0) \rightarrow \modulo{(\mu_{M^{\mathrm{CR}}})^{-1}(0)}{L} $$
defines a unique Cauchy-Riemann structure on 
$\modulo{(\mu_{M^{\mathrm{CR}}})^{-1}(0)}{L} $,
$(\varrho_{\mathrm{red}})^{-1}(0)$ can be regarded as 
the Cauchy-Riemann reduction of $\varrho^{-1}(0)$
with respect to $L$.
\end{enumerate}
\end{proof}

\bigskip

\begin{rem}
In particular, the % $n$-dimensional 
contact manifold $(M, \eta)$
is embedded in the $(2n-1)$-dimensional 
contact and Cauchy-Riemann manifold $(M^{\mathrm{CR}}, \eta^{\mathrm{CR}})$
with the contact form 
%$\eta^{\mathrm{CR}} = d^c \varrho \vert_{M^{\mathrm{CR}}}$.
$\eta^{\mathrm{CR}} = (\iota_{M^{\mathrm{CR}}})^{\ast}(d^c \varrho)$.
\end{rem}

\bigskip

The following proposition summarizes the results on 
the compatibility of the respective reductions.

\bigskip

\begin{prop}\label{compatibility-Kahler}
Let $L$ be an extendable Lie group and 
$L \times M^c \rightarrow M^c$ 
a free and proper action that extends 
$L \times M \rightarrow M$ and leaves 
$\varrho : M^c \rightarrow \mathbb R$ invariant.
%Then the inclusions of Corollary \ref{cor-inclusion} induce a 
Then there is the following commutative diagram
\[ \begin{array}{rcccccl}   
 (\mu_{M})^{-1} (0) \phantom{xx}
     & \hookrightarrow &
  (\mu_{M^{\mathrm{CR}}})^{-1} (0) \phantom{xx}
     & \hookrightarrow &
 (\mu_{M^c})^{-1} (0) \phantom{xx} \\
   & & \\
  \Big \downarrow \mbox{ }^{ \pi_{(\mu_M)^{-1}(0)} } & & \phantom{x}
  \Big \downarrow \mbox{ }^{ \pi_{(\mu_{M^{\mathrm{CR}}})^{-1}(0)} } & & \phantom{x} 
  \Big \downarrow \mbox{ }^{ \pi_{(\mu_{M^c})^{-1}(0)} } \\
   & & \\
\modulo{(\mu_{M})^{-1} (0)}{L} 
     & \hookrightarrow &
 \modulo{ (\mu_{M^{\mathrm{CR}}})^{-1} (0)}{L} 
     & \hookrightarrow &
 \modulo{(\mu_{M^c})^{-1} (0)}{L}
\end{array} \] 
%between manifolds.
%of differentiable manifolds.
of smooth maps.
\end{prop}

%\bigskip

\begin{proof}[{\bf Proof}]
The momentum zero levels 
$(\mu_{M})^{-1} (0)$, 
$(\mu_{M^{\mathrm{CR}}})^{-1} (0)$ and 
$(\mu_{M^c})^{-1} (0)$ 
are smooth because the $L$-orbits have constant dimensions.
Since the three actions of $L$ are proper and free, 
the three quotients 
$$ \modulo{(\mu_{M})^{-1} (0)}{L} \mbox{ and } 
   \modulo{ (\mu_{M^{\mathrm{CR}}})^{-1} (0)}{L} \mbox{ and } 
   \modulo{(\mu_{M^c})^{-1} (0)}{L} $$
are differentiable manifolds and the natural projections 
%\begin{itemize}
%\item $\pi_{(\mu_M)^{-1}(0)} : (\mu_{M})^{-1} (0) \rightarrow \modulo{(\mu_{M})^{-1} (0)}{L}$
%\item $ \pi_{(\mu_{M^{\mathrm{CR}}})^{-1}(0)} : (\mu_{M^{\mathrm{CR}}})^{-1} (0) \rightarrow
%             \modulo{ (\mu_{M^{\mathrm{CR}}})^{-1} (0)}{L}$ 
%\item $\pi_{(\mu_{M^c})^{-1}(0)} : (\mu_{M^c})^{-1} (0) \rightarrow 
%         \modulo{(\mu_{M^c})^{-1} (0)}{L}$ 
%\end{itemize}
$$ \pi_{(\mu_M)^{-1}(0)} \mbox{ and } 
   \pi_{(\mu_{M^{\mathrm{CR}}})^{-1}(0)} \mbox{ and } 
   \pi_{(\mu_{M^c})^{-1}(0)} $$
are differentiable maps as well as the induced inclusions
$$ \modulo{(\mu_{M})^{-1} (0)}{L} 
    \phantom{x} \hookrightarrow \phantom{x}
  \modulo{ (\mu_{M^{\mathrm{CR}}})^{-1} (0)}{L} 
    \phantom{x} \hookrightarrow \phantom{x} 
 \modulo{(\mu_{M^c})^{-1} (0)}{L} .$$
\end{proof}

%The situation can be described with this commutative diagramme:

%\begin{figure}[h!]
%\begin{center}
%\epsfxsize=\textwidth
%\input{diagramm5.pstex_t}
%\end{center}
%\end{figure}

%\begin{figure}[h!]
%  \begin{center}
%    \leavevmode
%    \epsfxsize=0.5\textwidth
%    \epsffile{menge-so-5.eps}
%  \end{center}
%\caption{Die Menge  $\pi( \lbrace \det(D \pi)=0 \rbrace)$ im Fall
%           $G = \mathrm{SO}_5$}
%\end{figure}

\bigskip

\subsection{Compatibility of reduced strata}\label{Compatibility of reduced strata}

\bigskip

The results of the %previous chapter 
Subsections \ref{Compatibility of moment maps for free actions} 
and \ref{CR-contact-Kahler-red}
are now applied to general proper actions 
on contact manifolds $(M, \eta)$
and their complexifications $M^c$.
Let $H$ be a compact subgroup of $G$. 
%Sjamaar (\cite{Sjamaar}) examines 
The isotropy types of $H$ define a stratification of $M$ (\cite{Sjamaar}).
The stratum
$ (M^c)_{(H)} = \lbrace x \in M^c \vert 
                     \exists g_0 \in G :  g_0 G_x g_0^{-1} = H \rbrace $
of points in $M^c$ with isotropy type $H$
%These manifolds 
is $G$-invariant and contains the complex submanifold
$M^c_H = \lbrace x \in M^c \vert G_x = H \rbrace.$
Then
$$ \modulo{(\mu_{M^c})^{-1}(0) \cap M^c_{(H)}}{G} \cong 
   \modulo{(\mu_{M^c})^{-1}(0) \cap M^c_{H}}{L},$$
where $L = \modulo{N_G(H)}{H}$ acts freely
(\cite{Greb-Heinzner}, \cite{Sjamaar}).
Proposition \ref{prop-contact-reduction-varrho}
implies that it is a K\"ahler manifold.
To abbreviate, define
$$\momentMcKH := (\mu_{M^c})^{-1}(0) \cap M^c_{(H)} \mbox{ and } 
  \momentMcH := (\mu_{M^c})^{-1}(0) \cap M^c_{H}$$
and similarly, in the contact case, 
$$\momentMKH := (\mu_{M})^{-1}(0) \cap M_{(H)} \mbox{ and }
  \momentMH := (\mu_{M})^{-1}(0) \cap M_{H}.$$
For future reference, 
the necessary facts for the K\"ahlerian reduction along the strata
$\momentMcKH$ are summarized here;
%The stratification results used here 
they are well known 
(%see \cite{Ammon}, 
   %  \cite{Guillemin-Sternberg}, 
     \cite{HHL}, 
     \cite{Lerman-Willett}, \cite{Willett}, 
     \cite{Sjamaar}).
%
%
%\begin{prop}[{\cite{Sjamaar}}]\label{KN-cap-conj-isotropy} 
%\hfill
\begin{enumerate}
\item 
Let $x_0 \in (\mu_{M^c})^{-1}(0)$ 
and let $V$ be the orthogonal complement to  
%$T_{x_0}(\orbit{G}{x_0}) \oplus J T_{x_0}(\orbit{G}{x_0}) $
the tangent space of the local $G^{\mathbb C}$-orbit through $x_0$
with respect to the K\"ahlerian metric.     
The momentum zero level along the stratum $(M^c)_{(H)}$, i.e., 
%$ (\mu_{M^c})^{-1}(0) \cap (M^c)_{(H)} $, 
$\momentMcKH$,
is locally and equivariantly isomorphic to $G \times^{H} V_H $ 
     % $ (\mu_{M^c})^{-1}(0) \cap (M^c)_{(H)} \cong G \times^{H} V_H ,$ 
     % $ (\mu_{(M^c)_{(H)}})^{-1}(0) \cong G \times^{H} V_H ,$ 
      where 
      $V_H = \lbrace v \in V \vert h \cdot v = v \mbox{ for all } h \in H \rbrace$.
The K\"ahlerian reduced space 
%$ \modulo{(\mu_{M^c})^{-1}(0) \cap (M^c)_{(H)}}{G}$ 
$\modulo{\momentMcKH}{G}$
is locally homeomorphic to $ V_H $
    %  $ \modulo{(\mu_{M^c})^{-1}(0) \cap (M^c)_{(H)}}{G} \cong V_H .$
    %  $ \modulo{(\mu_{(M^c)_{(H)}})^{-1}(0)}{G} \cong V_H .$
(\cite{Sjamaar}).
\item
The Kempf-Ness reduced space
$ \modulo{(\mu_{M^c})^{-1}(0)}{G} $
can be stratified into the strata
%$$\modulo{(\mu_{M^c})^{-1}(0) \cap (M^c)_{(H)}}{G} ,$$
%$$ \modulo{(\mu_{(M^c)_{(H)}})^{-1}(0)}{G} ,$$
$\modulo{\momentMcKH}{G} ,$
which inherit a natural symplectic and complex structure.
\end{enumerate}
Proposition \ref{prop-contact-reduction-varrho} 
can be applied
to the free action of $L := \modulo{N_G(H)}{H}$ 
on 
%$ (\mu_M)^{-1}(0) \cap M_H$.
%$ (\mu_{M_H})^{-1}(0) $.
$\momentMH$, 
where $N_G(H)$ is the normalizer of $H$ in $G$.
%
%
%
%
%\begin{prop}[\cite{Willett}, \cite{Lerman-Willett}]
%\label{prop-Willett}
%Let $(M, \eta)$ be a contact manifold 
In the case of a contact manifold $(M, \eta)$
on which $G$ acts in a proper fashion 
by contact transformations, 
recall the following facts (\cite{Willett}, \cite{Lerman-Willett}):
\begin{enumerate}
\item[c)] %(\cite{Willett}, Proposition 4.1)
%Then
The stratum
$ M_H = \lbrace m \in M \vert G_m = H \rbrace $
is a contact manifold and for the stratum
$ M_{(H)} = \lbrace m \in M \vert G_m \mbox{ is conjugate to } H \rbrace $
the quotients
%$$ \modulo{(\mu_M)^{-1}(0) \cap M_H}{L} 
%  = \modulo{(\mu_M)^{-1}(0) \cap M_{(H)}}{G}  $$
%$$ \modulo{(\mu_{M_H})^{-1}(0)}{\modulo{N_G(H)}{H}} 
%   \cong \modulo{(\mu_{M_{(H)}})^{-1}(0)}{G} $$
$$ \modulo{\momentMH}{L} = \modulo{\momentMKH}{G}$$
are naturally isomorphic manifolds, 
where $L = \modulo{N_G(H)}{H}$ acts freely and properly.
%
%Furthermore,
\item[d)] %(\cite{Willett}, \cite{Lerman-Willett})
The manifold $\modulo{\momentMH}{L} $ 
carries a uniquely induced contact form $\eta_{\mathrm{red}}$
with the property
%%$$ (\iota_{(\mu_M)^{-1}(0) \cap M_H \hookrightarrow M})^{\ast} (\eta) 
%%  = (\pi_{(\mu_M)^{-1}(0) \cap M_H \rightarrow \mu^{-1}(0) \cap M_H / N_G(H)})^{\ast} 
%%     (\eta_{\mathrm{red}}) ,$$
%$$ (\iota_{(\mu_M)^{-1}(0) \cap M_H})^{\ast} (\eta) 
%  = (\pi_{(\mu_M)^{-1}(0) \cap M_H})^{\ast} 
%     (\eta_{\mathrm{red}}) ,$$
%$$ (\iota_{(\mu_{M_H})^{-1}(0)})^{\ast} (\eta) 
%  = (\pi_{(\mu_{M_H})^{-1}(0)})^{\ast} 
%     (\eta_{\mathrm{red}}) ,$$
$$ (\iota_{\momentMH})^{\ast} (\eta) 
  = (\pi_{\momentMH})^{\ast} 
     (\eta_{\mathrm{red}}) ,$$
where
%$\iota_{(\mu_M)^{-1}(0) \cap M_H} : (\mu_M)^{-1}(0) \cap M_H \hookrightarrow M $ 
%$\iota_{(\mu_{M_H})^{-1}(0)} : (\mu_{M_H})^{-1}(0) \hookrightarrow M $ 
$\iota_{\momentMH} : \momentMH \hookrightarrow M $ 
and 
%$\pi_{(\mu_M)^{-1}(0) \cap M_H}$ 
$\pi_{\momentMH}$
is the projection of 
%$(\mu_M)^{-1}(0) \cap M_H$ to $ \modulo{\mu^{-1}(0) \cap M_H}{L} .$
$\momentMH$ to $ \modulo{\momentMH}{L} .$
%$$\pi_{(\mu_M)^{-1}(0) \cap M_H} : (\mu_M)^{-1}(0) \cap M_H \rightarrow 
%                                  \modulo{\mu^{-1}(0) \cap M_H}{L} .$$
%$$\pi_{(\mu_{M_H})^{-1}(0)} : (\mu_{M_H})^{-1}(0) \rightarrow 
%                                  \modulo{(\mu_{M_H})^{-1}(0)}{\modulo{N_G(H)}{H}} .$$
\end{enumerate}
%\hfill $\Box$
%\end{prop}
%
%
%\begin{proof}[{\bf Proof}]
%Part a) is proved in \cite{Willett}.
%The result implies that part b) follows from Proposition \ref{free-proper-case}.
%\end{proof}

Note that these facts treat every stratum independently and
one obtains for each stratum 
%$M_{(H)} \cap (\mu_M)^{-1}(0)$
$\momentMKH$
of $(\mu_M)^{-1}(0)$ a reduced contact space;
there is no %structure 
condition %in \cite{Willett} 
that links the various contact structures.

\bigskip

\begin{lem}
If $x_0 \in (\mu_M)^{-1}(0)$ and $H = G_{x_0}$
then $(M^c)_H$ complexifies $M_H$.
\end{lem}

%\bigskip

\begin{proof}[{\bf Proof}]
%Embed a $G$-invariant neighbourhood 
%$U(x_0) \hookrightarrow G^{\mathbb C} \times^{H^{\mathbb C}} V$
%as described in Proposition \ref{slice-model}.
%In particular, $U(x_0)$ is embedded $G$-equivariantly and openly
%in a $G^{\mathbb C}$-manifold.
%
%
Let $V$ be the complex vector subspace in $T_{x_0} M^c$,
which is the complement with respect to the K\"ahlerian metric
of the local $G^{\mathbb C}$-orbit through $x_0$.
There is a $G$-invariant neighbourhood $U(x_0)$ of $x_0$
which is openly and $G$-equivariantly embedded in the complex 
$G^{\mathbb C}$-manifold $G^{\mathbb C} \times^{H^{\mathbb C}} V$ 
(\cite{Heinzner-Kurtdere}, \cite{Kurtdere}).
Since $x_0 \in (\mu_M)^{-1}(0) \subset (\mu_{M^c})^{-1}(0)$,
$H^{\mathbb C} = (G_{x_0})^{\mathbb C} = (G^{\mathbb C})_{x_0}$
and it is possible to assume in addition that
$V = W^{\mathbb C}$, where $W \subset T_{x_0} M$
is an $H$-invariant subspace such that
$M \cap U(x_0)$ embeds openly in $G \times^H W$.
If 
$ W_{\langle H \rangle} 
     = \lbrace w \in W \vert \orbit{h}{w} = h \mbox{ for all } h \in H \rbrace $,
it follows that for
$ V_{\langle H \rangle} 
     = \lbrace v \in V \vert \orbit{h}{v} = v \mbox{ for all } h \in H \rbrace 
     = (W_{\langle H \rangle})^{\mathbb C} $.
Finally
$ U(x_0) \cap M_H \hookrightarrow G \times^H W_{\langle H \rangle} $
and 
$ U(x_0) \cap M^c_H \hookrightarrow 
  G^{\mathbb C} \times^{H^{\mathbb C}} (W_{\langle H \rangle})^{\mathbb C} $
are open embeddings.
Since $ G^{\mathbb C} \times^{H^{\mathbb C}} (W_{\langle H \rangle})^{\mathbb C}$
can be regarded as the complexification of
$ G \times^{H} W_{\langle H \rangle}$,
this proves the claim.
\end{proof}

\bigskip

The set
$ (M^c)_H = \lbrace z \in M^c \vert G_z = H \rbrace $
is a complex submanifold of $M^c$
(\cite{Sjamaar}). 
%{\bf Reyer Sjamaar, Holomorphic slices, symplectic reduction and multiplicities of representations, p. 23.)} 
The normalizer $N_G(H)$ of $H$ in $G$ acts naturally on $(M^c)_H$.
The induced action of $L:= \modulo{N_G(H)}{H}$ on $(M^c)_H$ is free.
If $x_0 \in (\mu_M)^{-1}(0)$ and $H = G_{x_0}$, 
it follows that
$M_H \hookrightarrow (M^c)_H$
is embedded in a K\"ahlerian submanifold of $M^c$.
The $1$-form
$ (\iota_{M_H \hookrightarrow M})^{\ast}(\eta) $
is the contact $1$-form on $M_H$ and
$ (\iota_{(M^c)_H \hookrightarrow M^c})^{\ast}(-dd^c \varrho) $
is the K\"ahlerian form on $M^c$.
Furthermore,
$ \eta = (\iota_{M \hookrightarrow M^c})^{\ast}(d^c \varrho) $
and
\begin{align*}
 (\iota_{M_H \hookrightarrow M})^{\ast} \eta 
 & = (\iota_{M_H \hookrightarrow M})^{\ast}
     ((\iota_{M \hookrightarrow M^c})^{\ast}(d^c \varrho)) \\
% & = (\iota_{M \hookrightarrow M^c} \circ 
%      \iota_{M_H \hookrightarrow M})^{\ast}(d^c \varrho) \\
% & = (\iota_{(M^c)_H \hookrightarrow M^c} \circ 
%      \iota_{M_H \hookrightarrow (M^c)_H})^{\ast}
%     (d^c \varrho) \\
 & = (\iota_{M_H \hookrightarrow (M^c)_H})^{\ast}
     ((\iota_{(M^c)_H \hookrightarrow M^c})^{\ast} (d^c \varrho)) \\
 & = (\iota_{M_H \hookrightarrow (M^c)_H})^{\ast} (d^c (\varrho \vert_{(M^c)_H})) .
\end{align*}
Thus the contact manifold
$(M_H, (\iota_{M_H \hookrightarrow M})^{\ast}(\eta))$
is $N_G(H)$-equivariantly embedded in the K\"ahler manifold
$((M^c)_H, -dd^c \varrho \vert_{(M^c)_H}), $
where $\varrho \vert_{(M^c)_H} : (M^c)_H \rightarrow \mathbb R$
is a strictly plurisubharmonic and $N_G(H)$-invariant function.
%
%
%
%The stratifications $M_H$ and $(M^c)_H$ are compatible;
%and $(M^c)_H$ can be regarded as a complexification of $M_H$.
The following result establishes 
the compatibility of the extensions with reductions.

\bigskip

\begin{prop}\label{complex-contact-stratum}
Let $\varrho_{\mathrm{red}} : \modulo{(\mu_{M^c})^{-1}(0)}{G} \rightarrow \mathbb R$
be the function defined by 
$\varrho_{\mathrm{red}} \circ \pi_{(\mu_{M^c})^{-1}(0)} 
   = \varrho \circ \iota_{(\mu_{M^c})^{-1}(0)}$.
Let $x_0 \in (\mu_M)^{-1}(0)$ and $H = G_{x_0}$ be the isotropy group.
The contact manifold
$(M_H, (\iota_{M_H \hookrightarrow M})^{\ast}(\eta))$
embeds in the K\"ahlerian manifold
$((M^c)_H, (\iota_{(M^c)_H \hookrightarrow M^c})^{\ast}(dd^c \varrho))$
such that
$ (\iota_{M_H \hookrightarrow (M^c)_H})^{\ast}(d^c \varrho)
   = (\iota_{M_H \rightarrow M})^{\ast}(\eta). $
The reduced spaces
%$$ \Big( \modulo{(\mu_{M})^{-1}(0) \cap M_H}{L}, 
%   (\iota_{M_H \hookrightarrow M})^{\ast}(\eta_{\mathrm{red}}) \Big)
%    \mbox{ and } 
%   \Big( \modulo{(\mu_{M^c})^{-1}(0) \cap (M^c)_H}{L}, \varrho_{\mathrm{red}} \Big) $$
%$$ \Big( \modulo{(\mu_{M_H})^{-1}(0)}{\modulo{N_G(H)}{H}}, 
%   (\iota_{M_H \rightarrow M})^{\ast}(\eta_{\mathrm{red}}) \Big)
%    \mbox{ and } 
%   \Big( \modulo{(\mu_{M})^{-1}(0)}{\modulo{N_G(H)}{H}}, \varrho_{\mathrm{red}} \Big) $$
$$ \Big( \modulo{\momentMH}{L}, 
   (\iota_{M_H \hookrightarrow M})^{\ast}(\eta_{\mathrm{red}}) \Big)
    \mbox{ and } 
   \Big( \modulo{\momentMcH}{L}, \varrho_{\mathrm{red}} \vert_{\momentMcH / L} \Big), $$
where $L = \modulo{N_G(H)}{H}$,
are related by the property that for the induced embedding
%$$ \iota_{(\mu_{M_H})^{-1}(0) / (N_G(H) / H)}:
%   \modulo{(\mu_{M_H})^{-1}(0)}{\modulo{N_G(H)}{H}} \hookrightarrow 
%           \modulo{(\mu_{(M^c)_H})^{-1}(0)}{\modulo{N_G(H)}{H}} $$
%$ \iota_{(\mu_{M})^{-1}(0) \cap M_H / L} $ of 
%$ \modulo{(\mu_{M})^{-1}(0) \cap M_H}{L}$ into $\modulo{(\mu_{M^c})^{-1}(0) \cap (M^c)_H}{L}$
$ \iota_{\momentMH / L} : \modulo{\momentMH}{L} \hookrightarrow \modulo{\momentMcH}{L}$
%by the property that
%$$ \big( \iota_{(\mu_{M})^{-1}(0) \cap M_H / L} \big)^{\ast} 
%   (d^c \varrho_{\mathrm{red}}) 
% = \big( \iota_{M_H \hookrightarrow M} \big)^{\ast} (\eta_{\mathrm{red}}). $$
$$ \big( \iota_{\momentMH / L} \big)^{\ast} 
   (d^c \varrho_{\mathrm{red}} \vert_{\momentMcH / L}) 
 = \big( \iota_{M_H \hookrightarrow M} \big)^{\ast} (\eta_{\mathrm{red}}). $$
\end{prop}

%\bigskip

\begin{proof}[{\bf Proof}]
Since $L = \modulo{N_G(H)}{H}$ acts freely on $(M^c)_H$,
Proposition \ref{prop-contact-reduction-varrho}
applies. 
It follows that the contact moment map
$ \mu_{M_H} : M_H \rightarrow \mathrm{Lie}(L)^{\ast} $
and the K\"ahlerian moment map 
$ \mu_{(M^c)_H} : (M^c)_H \rightarrow \mathrm{Lie}(L)^{\ast} $
%$ \mu_{(M^c)_H} : (M^c)_H \rightarrow \mathrm{Lie}(N_G(H))^{\ast} $
define smooth momentum zero levels 
%$(\mu_{M_H})^{-1}(0)$ and $(\mu_{(M^c)_H})^{-1}(0)$
$\momentMH$ and $\momentMcH$,
because %$N_G(H)$ 
$L$ acts freely.
For the same reason
%$\modulo{(\mu_{(M^c)_H})^{-1}(0)}{L}$ and $\modulo{(\mu_{M_H})^{-1}(0)}{L}$
%$\modulo{(\mu_{(M^c)_H})^{-1}(0)}{N_G(H)}$ and $\modulo{(\mu_{M_H})^{-1}(0)}{N_G(H)}$
$\modulo{\momentMcH}{L}$ and $\modulo{\momentMH}{L}$
are smooth, and the restriction of 
%$\varrho \vert_{(\mu_{(M^c)_H})^{-1}(0)}$
$\varrho \vert_{\momentMcH}$
defines a strictly plurisubharmonic function
%$$ \varrho_{\mathrm{red}} \vert_{(\mu_{(M^c)_H})^{-1}(0)} $$
$ \varrho_{\mathrm{red}} \vert_{\momentMcH} $
such that the embedding
%$ \iota_{(\mu_{M_H})^{-1}(0) / L} $ 
%of $\modulo{(\mu_{M_H})^{-1}(0)}{L}$ into $\modulo{(\mu_{(M^c)_H})^{-1}(0)}{L} $
%$ \iota_{(\mu_{M_H})^{-1}(0) / L} : 
%       \modulo{(\mu_{M_H})^{-1}(0)}{L} \hookrightarrow \modulo{(\mu_{(M^c)_H})^{-1}(0)}{L} $
%$$ \iota_{(\mu_{M_H})^{-1}(0) / N_G(H)} :
%    \modulo{(\mu_{M_H})^{-1}(0)}{N_G(H)} \hookrightarrow 
%    \modulo{(\mu_{(M^c)_H})^{-1}(0)}{N_G(H)} $$
$ \iota_{\momentMH / L} : 
       \modulo{\momentMH}{L} \hookrightarrow \modulo{\momentMcH}{L} $
satisfies
%$ (\iota_{(\mu_{M_H})^{-1}(0) / L})^{\ast}
%   (d^c \varrho_{\mathrm{red}}) 
% = (\iota_{M_H \hookrightarrow M})^{\ast}(\eta)_{\mathrm{red}} .$
$ (\iota_{\momentMH / L})^{\ast}
   (d^c \varrho_{\mathrm{red}} \vert_{\momentMcH / L}) 
 = (\iota_{M_H \hookrightarrow M})^{\ast}(\eta_{\mathrm{red}}) .$
\end{proof}

\bigskip

In the same way, the strata $M_{(H)}$ and $(M^c)_{(H)}$
are compatible.
For every closed subgroup $H$ of $G$ the strata
%$ M_{(H)} \cap \modulo{(\mu_{M})^{-1}(0)}{G}$
%$ \modulo{(\mu_{M_{(H)}})^{-1}(0)}{G}$ 
$ \modulo{\momentMKH}{G}$ 
of $\modulo{(\mu_{M})^{-1}(0)}{G}$ and
%$ (M^c)_{(H)} \cap \modulo{(\mu_{M^c})^{-1}(0)}{G}$
%$ \modulo{(\mu_{(M^c)_{(H)}})^{-1}(0)}{G}$
$ \modulo{\momentMcKH}{G}$
of $\modulo{(\mu_{M^c})^{-1}(0)}{G}$ are compatible by the function
$\varrho_{\mathrm{red}} \vert : \modulo{(\mu_{M^c})^{-1}(0)}{G} \rightarrow \mathbb R$ 
induced by the restriction
$ \varrho \vert_{(\mu_{M^c})^{-1}(0)} $
in the following sense:

\bigskip

\begin{prop}\label{complex-contact-stratum-red}
The embedding 
%$\iota_{M \hookrightarrow M^c} : M \hookrightarrow M^c$ 
$M \hookrightarrow M^c$ 
induces %on each stratum 
embeddings
%
%%$$  M_{(H)} \cap (\mu_{M})^{-1}(0) \hookrightarrow (M^c)_{(H)} \cap (\mu_{M^c})^{-1}(0) $$
%$$ (\mu_{M_{(H)}})^{-1}(0) \hookrightarrow (\mu_{(M^c)_{(H)}})^{-1}(0) $$ 
%and 
%%$$ \modulo{M_{(H)} \cap (\mu_{M})^{-1}(0)}{G} \hookrightarrow \modulo{(M^c)_{(H)} \cap (\mu_{M^c})^{-1}(0)}{G} $$
%$$ \modulo{(\mu_{M_{(H)}})^{-1}(0)}{G} \hookrightarrow \modulo{(\mu_{(M^c)_{(H)}})^{-1}(0)}{G} $$
%
%$$ (\mu_{M_{(H)}})^{-1}(0) \hookrightarrow (\mu_{(M^c)_{(H)}})^{-1}(0) 
%   \mbox{ and } 
% \modulo{(\mu_{M_{(H)}})^{-1}(0)}{G} \hookrightarrow \modulo{(\mu_{(M^c)_{(H)}})^{-1}(0)}{G} $$
$$ \momentMKH \hookrightarrow \momentMcKH 
   \mbox{ and } 
 \modulo{\momentMKH}{G} \hookrightarrow \modulo{\momentMcKH}{G} $$
such that a contact manifold 
%$ \modulo{M_{(H)} \cap (\mu_{M})^{-1}(0)}{G}$ 
%$ \modulo{(\mu_{M_{(H)}})^{-1}(0)}{G}$
$ \modulo{\momentMKH}{G}$
embeds in %the K\"ahlerian manifold 
%$\modulo{(M^c)_{(H)} \cap (\mu_{M})^{-1}(0)}{G}$ 
%$\modulo{(\mu_{(M^c)_{(H)}})^{-1}(0)}{G}$
$\modulo{\momentMcKH}{G}$
and has the property
%$$ (\iota_{ M_{(H)} \cap (\mu_{M})^{-1}(0) \hookrightarrow 
%           (M^c)_{(H)} \cap (\mu_{M})^{-1}(0) })^{\ast} 
%  (d^c \varrho) = \eta_{\mathrm{red}} , $$
%$$ (\iota_{ (\mu_{M_{(H)}})^{-1}(0) / G \hookrightarrow 
%           (\mu_{(M^c)_{(H)}})^{-1}(0) / G })^{\ast} 
%  (d^c \varrho_{\mathrm{red}}) = \eta_{\mathrm{red}} , $$
$$ (\iota_{ \momentMKH / G \hookrightarrow 
           \momentMcKH / G })^{\ast} 
  (d^c \varrho_{\mathrm{red}} \vert_{\momentMcKH / G}) = \eta_{\mathrm{red}} , $$
where $\eta_{\mathrm{red}}$ is the reduced contact form on 
%$\modulo{M_{(H)} \cap (\mu_{M})^{-1}(0)}{G}$. 
%$\modulo{ (\mu_{M_{(H)}})^{-1}(0)}{G} $.
$\modulo{ \momentMKH}{G} $.
\end{prop}

%\bigskip

\begin{proof}[{\bf Proof}]
%Willett's result (\cite{Willett}, see here Proposition \ref{prop-Willett} a))
%shows that
The quotients
%$$ \modulo{M_{H} \cap (\mu_{M})^{-1}(0)}{N_G(H)} 
% \cong \modulo{M_{(H)} \cap (\mu_{M})^{-1}(0)}{G}  $$
%$$ \modulo{ (\mu_{M_{H}})^{-1}(0)}{N_G(H)} \cong \modulo{ (\mu_{M_{(H)}})^{-1}(0)}{G} $$
%$ \modulo{ (\mu_{M_{H}})^{-1}(0)}{L} $ and $ \modulo{ (\mu_{M_{(H)}})^{-1}(0)}{G} $
$ \modulo{ \momentMH}{L} $ and $ \modulo{ \momentMKH}{G} $
are naturally diffeomorphic (\cite{Willett}). 
%(\cite{Willett}, see here Proposition \ref{prop-Willett} a)).
%Sjamaar and Lerman prove that
Similarly
%$$ \modulo{(M_{H})^c \cap (\mu_{M^c})^{-1}(0)}{N_G(H)} 
% \cong \modulo{(M^c)_{(H)} \cap (\mu_{M^c})^{-1}(0)}{G} $$
%$$ \modulo{(\mu_{(M^c)_{H}})^{-1}(0)}{N_G(H)} 
% \cong \modulo{(\mu_{(M^c)_{(H)}})^{-1}(0)}{G} $$
%$ \modulo{(\mu_{(M^c)_{H}})^{-1}(0)}{L}$ and $ \modulo{(\mu_{(M^c)_{(H)}})^{-1}(0)}{G} $
$ \modulo{\momentMcH}{L}$ and  $\modulo{\momentMcKH}{G} $
are naturally diffeomorphic.
Proposition \ref{complex-contact-stratum} 
%$$ (\iota_{(\mu_{M_H})^{-1}(0) / G \rightarrow 
%           (\mu_{(M^c)_H})^{-1}(0) / G})^{\ast}
%   (d^c \varrho_{\mathrm{red}}) 
% = (\iota_{M_H \hookrightarrow M})^{\ast}(\eta_{\mathrm{red}}) $$
and the commutativity of the diagram
\[ \begin{array}{rcccl}
    % \modulo{(\mu_{M_H})^{-1}(0)}{L} 
     \modulo{\momentMH}{L}  %= \modulo{(\mu_M)^{-1}(0) \cap M_H}{L} 
        & \hookrightarrow &  
     %\modulo{(\mu_{(M^c)_H})^{-1}(0)}{L} 
      \modulo{\momentMcH}{L}\\ 
        %= \modulo{(\mu_{M^c})^{-1}(0) \cap (M^c)_H}{L} \\
        & & \\
     \Big \downarrow \cong \phantom{xx} & &  \Big \downarrow \cong \\
        & & \\
    %\modulo{(\mu_{M_{(H)}})^{-1}(0)}{G} 
      \modulo{\momentMKH}{G} %= \modulo{(\mu_M)^{-1}(0) \cap M_{(H)}}{G} 
        & \hookrightarrow &  
     %\modulo{(\mu_{(M^c)_{(H)}})^{-1}(0)}{G} 
      \modulo{\momentMcKH}{G} %= \modulo{(\mu_{M^c})^{-1}(0) \cap (M^c)_{(H)}}{G} 
   \end{array}  \]
%where 
%$\iota_{(\mu_{M_H})^{-1}(0) / G \rightarrow (\mu_{(M^c)_H})^{-1}(0) / G} : 
% \modulo{(\mu_M)^{-1}(0) \cap M_{(H)}}{G} 
%  \rightarrow \modulo{(\mu_{M^c})^{-1}(0) \cap (M^c)_{(H)}}{G} $
prove the claim.
\end{proof}

\bigskip

As a summary, 
the geometry of the contact, Cauchy-Riemann and K\"ahlerian reductions
can be described as follows.

\bigskip

\begin{cor}
Let $G$ be an extendable Lie group and
$(M, \eta)$ a proper $G$-contact manifold.
There is an equivariant Stein complexification $M^c$
with a smooth strictly plurisubharmonic function
$\varrho : M^c \rightarrow \mathbb R$
such that $d^c \varrho$ extends $\eta$.
Let the fuction
$\varrho_{\mathrm{red}} : \modulo{(\mu_{M^c})^{-1}(0)}{G} \rightarrow \mathbb R$
be defined %the induced 
%strictly plurisubharmonic 
%function such that
by
$\varrho \circ \iota_{(\mu_{M^c})^{-1}(0)} 
   = \varrho_{\mathrm{red}} \circ \pi_{(\mu_{M^c})^{-1}(0)}$.
%where 
%$\pi_{(\mu_{M^c})^{-1}(0)} : (\mu_{M^c})^{-1}(0) \rightarrow \modulo{(\mu_{M^c})^{-1}(0)}{G}$.
Let $H < G$ be a compact subgroup and $L = \modulo{N_G(H)}{H}$.
Then for every stratum $(M^c)_{(H)}$,
the $1$-form $d^c (\varrho_{\mathrm{red}} \vert_{\momentMcKH / G} )$ 
\begin{enumerate}
\item %to $\modulo{(\mu_{(M^c)_{(H)}})^{-1}(0)}{G}$
      %   $\modulo{\momentMcKH}{G}$
      provides the K\"ahlerian reduction of $(M^c)_{(H)}$,
\item %to $(\varrho_{\mathrm{red}})^{-1}(0) \cap \modulo{(\mu_{(M^c)_{(H)}})^{-1}(0)}{G}$
      pulled-back to
      $(\varrho_{\mathrm{red}})^{-1}(0) \cap \big( \modulo{\momentMcKH}{G} \big )$
      is equivalent to the contact and the Cauchy-Riemann reduction
      of $(M^{\mathrm{CR}})_H := (\varrho \vert_{(M^c)_H})^{-1}(0) \subset (M^c)_H$ 
\item and pulled back to 
      %$\modulo{(\mu_{(M)_{(H)}})^{-1}(0)}{G}$
      $\modulo{\momentMKH}{G}$
      provides the contact reduction of $M_{(H)}$.
\end{enumerate}
\end{cor}

%\medskip

\begin{proof}[{\bf Proof}]
The compatibility of the reductions 
with the K\"ahlerian reduction is shown in 
Proposition \ref{compatibility-Kahler}.
The Cauchy-Riemann reduction is carried out in 
Proposition \ref{contact-isomorphic} 
%and Proposition \ref{CR-isomorphic},
and the contact reduction by 
Proposition \ref{prop-contact-reduction-varrho}.
These results are applied to the stratifications 
described in 
Proposition \ref{complex-contact-stratum} and 
Proposition \ref{complex-contact-stratum-red}.
\end{proof}

\bigskip

\medskip

{\bf Piecewise contact structures}

\bigskip

For a proper action
$G \times M \rightarrow M$ on a contact manifold $(M, \eta)$
by contact transformations 
$\modulo{(\mu_M)^{-1}(0)}{G}$
is stratified into smooth contact manifolds
%$ \modulo{M_{(H)} \cap (\mu_M)^{-1}(0)}{G} $
%$ \modulo{(\mu_{M_{(H)}})^{-1}(0)}{G} $ .
$ \modulo{\momentMKH}{G} $
(\cite{Willett}, \cite{Lerman-Willett}).
The respective contact structures are induced by the contact reductions.
However, these contact structures are treated separately;
the transition between two strata is not worked out 
in \cite{Willett} and \cite{Lerman-Willett}.
In the symplectic setting, 
a Poisson structure can be defined 
%on $\modulo{\mu^{-1}(0)}{G}$ 
which allows one to discuss 
the compatibility of the various %symplectic structures.
strata.
As a suitable tool in the case of contact manifolds,
the following definition of a piecewise contact structure
is suggested here to state a compatibility condition
from stratum to stratum.

\bigskip

\begin{defi}\label{def-piecewise-contact}
Let $M = \bigcup_{\alpha \in I} M_{\alpha}$
be a stratified topological space such that 
every stratum $M_{\alpha}$ is a differentiable manifold.
A family of $1$-forms $\eta_{\alpha}$ on $M_{\alpha}$, $\alpha \in I$,
is called a piecewise contact structure if
\begin{enumerate}
\item each $(M_{\alpha}, \eta_{\alpha})$ is a contact manifold,
\item there is a complex space $M^c$ with a stratification $M_{\alpha}^c$
      into complex manifolds such that an embedding
      $ \iota : M \hookrightarrow M^c $
      induces embeddings 
      $ \iota_{\alpha} : M_{\alpha} \hookrightarrow M_{\alpha}^c$
      as totally real submanifolds, %of half (real) dimension,
\item there is a strictly plurisubharmonic function 
      $\varrho : M^c \rightarrow \mathbb R$
      such that for every $\alpha$ the restricted function	
      $ \varrho \vert_{M_{\alpha}^c} $
      is smooth and satisfies
      $$ \iota_{\alpha}^{\ast} (d^c(\varrho \vert_{M_{\alpha}^c})) 
        = \eta_{\alpha} .$$
\end{enumerate}
\end{defi}

\bigskip

\begin{rem}
%Corollary \ref{G-times-K-invariante-function} 
Theorem \ref{Komplexifizierungssatz}
shows that a smooth contact manifold is a 
piecewise contact manifold which consists of one stratum only.
\end{rem}

\bigskip

In the case of an extendable Lie group (\cite{Heinzner-Kurtdere}, \cite{Kurtdere})
show that $\modulo{(\mu_{M^c})^{-1}(0)}{G}$
enherits from $M^c$ the structure of a complex space 
in a natural way on which $\varrho_{\mathrm{red}}$
is a strictly plurisubharomonic function.
Then 
$\modulo{(\mu_{M^c})^{-1}(0)}{G} = \bigcup_{H<G} \modulo{\mathcal{M}(M_{(H)}^c)}{G}$
stratifies the reduced space $\modulo{(\mu_{M^c})^{-1}(0)}{G}$
and Proposition \ref{complex-contact-stratum-red}
can now also be stated as follows:

\bigskip

\begin{theorem}\label{strata-patched-together}
Let $G$ be an extendable Lie group 
which acts properly on a contact manifold $(M, \eta)$
by contact transformations. 
%Then $\modulo{(\mu_M)^{-1}(0)}{G}$ can be given 
%naturally the structure of a piecewise contact manifold.
Then there is a canonically defined structure of 
a piecewise contact manifold on the quotient
$\modulo{(\mu_M)^{-1}(0)}{G}$.
\hfill $\Box$
\end{theorem}

\bigskip

%{\phantom{Hier sollte einfach etwas Text stehen.}}
%\hfill\eject

%\newpage

%\setcounter{page}{0}
%\pagenumbering{Roman}
%\setcounter{page}{1}

\addcontentsline{toc}{section}{\numberline{}\large{References}}

%\vspace{0.5cm}

%%Bochum, \today

\end{document}